\documentclass[11pt]{amsart}

\usepackage{enumerate}
\usepackage{amsmath,amsthm,verbatim,amssymb,amsfonts,amscd,amsopn,amsxtra,graphicx,lmodern,enumitem}
\usepackage{hyperref}
\usepackage{tikz-cd} 
\usepackage{graphics}
\graphicspath{ {./images/} }
\usepackage{mathrsfs}
\usepackage{relsize}
\usepackage{enumitem}
\usepackage[utf8]{inputenc}
\usepackage{csquotes}
\usepackage{amsthm}
\usepackage{thmtools}
\usepackage{mathtools}
\usepackage{microtype}
\usepackage{pdfrender,xcolor}
\usepackage[sort cites=true, backend=biber]{biblatex}
\usepackage{lua-visual-debug}
\usepackage{biblatex}
\usepackage[toc,page]{appendix}

\tikzset{
	symbol/.style={
		draw=none,
		every to/.append style={
			edge node={node [sloped, allow upside down, auto=false]{$#1$}}}
	}
}

\newlist{condenum}{enumerate}{1} 
\setlist[condenum]{label=\bfseries C\arabic*., 
	ref=\arabic*, wide}

\begin{document}
	\pdfrender{StrokeColor=black,TextRenderingMode=2,LineWidth=0.2pt}	
	
	\title{Purity and distances between conjugates of elements over henselian valued fields}
	
	\author{Arpan Dutta}
	\address{Department of Mathematics, School of Basic Sciences, IIT Bhubaneswar, Argul,
		Odisha, India, 752050.}
	\email{arpandutta@iitbbs.ac.in}
		\author[Novacoski]{Josnei Novacoski}
\address{Departamento de Matem\'{a}tica,         Universidade Federal de S\~ao Carlos, Rod. Washington Luís, 235, 13565--905, S\~ao Carlos -SP, Brazil}
\email{josnei@ufscar.br}


	\def\NZQ{\mathbb}               
	\def\NN{{\NZQ N}}
	\def\QQ{{\NZQ Q}}
	\def\ZZ{{\NZQ Z}}
	\def\RR{{\NZQ R}}
	\def\CC{{\NZQ C}}
	\def\AA{{\NZQ A}}
	\def\BB{{\NZQ B}}
	\def\PP{{\NZQ P}}
	\def\FF{{\NZQ F}}
	\def\GG{{\NZQ G}}
	\def\HH{{\NZQ H}}
	\def\UU{{\NZQ U}}
	\def\P{\mathcal P}
	
	%
	%
	\let\union=\cup
	\let\sect=\cap
	\let\dirsum=\oplus
	\let\tensor=\otimes
	\let\iso=\cong
	\let\Union=\bigcup
	\let\Sect=\bigcap
	\let\Dirsum=\bigoplus
	\let\Tensor=\bigotimes
	
	\theoremstyle{plain}
	\newtheorem{Theorem}{Theorem}[section]
	\newtheorem{Lemma}[Theorem]{Lemma}
	\newtheorem{Corollary}[Theorem]{Corollary}
	\newtheorem{Proposition}[Theorem]{Proposition}
	\newtheorem{Problem}[Theorem]{}
	\newtheorem{Conjecture}[Theorem]{Conjecture}
	\newtheorem{Question}[Theorem]{Question}
	
	\theoremstyle{definition}
	\newtheorem{Example}[Theorem]{Example}
	\newtheorem{Examples}[Theorem]{Examples}
	\newtheorem{Definition}[Theorem]{Definition}
	
	\theoremstyle{remark}
	\newtheorem{Remark}[Theorem]{Remark}
	\newtheorem{Remarks}[Theorem]{Remarks}

	\newcommand{\n}{\par\noindent}
	\newcommand{\nn}{\par\vskip2pt\noindent}
	\newcommand{\sn}{\par\smallskip\noindent}
	\newcommand{\mn}{\par\medskip\noindent}
	\newcommand{\bn}{\par\bigskip\noindent}
	\newcommand{\pars}{\par\smallskip}
	\newcommand{\parm}{\par\medskip}
	\newcommand{\parb}{\par\bigskip}

	\let\epsilon=\varepsilon
	\let\phi=\varphi
	\let\kappa=\varkappa
	
	\newcommand{\trdeg}{\mbox{\rm trdeg}\,}
	\newcommand{\rr}{\mbox{\rm rat rk}\,}
	\newcommand{\sep}{\mathrm{sep}}
	\newcommand{\ac}{\mathrm{ac}}
	\newcommand{\ins}{\mathrm{ins}}
	\newcommand{\res}{\mathrm{res}}
	\newcommand{\Gal}{\mathrm{Gal}\,}
	\newcommand{\ch}{\operatorname{char}}
	\newcommand{\Aut}{\mathrm{Aut}\,}
	\newcommand{\kras}{\mathrm{kras}\,}
	\newcommand{\dist}{\mathrm{dist}\,}
	\newcommand{\ord}{\mathrm{ord}\,}
	\newcommand{\Div}{\mathrm{Div}\,}
	\newcommand{\Supp}{\mathrm{Supp}\,}
	\newcommand{\Spec}{\mathrm{Spec}\,}
	\newcommand{\height}{\mathrm{ht}\,}
	\newcommand{\rk}{\mathrm{rank}\,}
	\newcommand{\Diff}{\mathrm{Diff}\,}
	\newcommand{\Ram}{\mathrm{Ram}\,}
	\newcommand{\id}{\mathrm{id}\,}
	\newcommand{\lex}{\mathrm{lex}\,}
	\newcommand{\gr}{\mathrm{gr}\,}
	\newcommand{\init}{\mathrm{in}\,}
	\newcommand{\depth}{\mathrm{depth}\,}

	\newtheorem{appendixlemma}{Lemma}[section]
	\newtheorem{appendixthm}[appendixlemma]{Theorem}

	\let\phi=\varphi
	\let\kappa=\varkappa
	
	\def \a {\alpha}
	\def \b {\beta}
	\def \s {\sigma}
	\def \d {\delta}
	\def \g {\gamma}
	\def \o {\omega}
	\def \l {\lambda}
	\def \th {\theta}
	\def \e {$\epsilon$}
	\def \D {\Delta}
	\def \G {\Gamma}
	\def \O {\Omega}
	\def \L {\Lambda}

	\def\op{\operatorname}
	\def\epm{\epsilon_\mu}
	\def\ps{\partial_s}
	\def\kx{K[x]}
	\def\kbx{\overline{K}[x]}
	\def\kb{\overline{K}}
	\def\omu{\overline\mu}
	\def\ggm{\mathcal{G}_\mu}
	\def\inm{\op{in}_\mu}
	\def\sub{\subseteq}
	\def\kpm{\op{KP}(\mu)}

	%
	%
	\textwidth=15cm \textheight=22cm \topmargin=0.5cm
	\oddsidemargin=0.5cm \evensidemargin=0.5cm \pagestyle{plain}


	
	\date{\today}
	
	\maketitle
	
	
	\begin{abstract}
		For a henselian valued field $(K,v)$ and a separable-algebraic element $a\in\overline{K}\setminus K$, we consider the set $S_K(a):= \{ v(a-a^\prime) \mid a^\prime\neq a \text{ is a $K$-conjugate of $a$} \}$. The central aim of this paper is to provide a bound for the cardinality of the set $S_K(a)$, and to characterize the elements $a$ for which this set is a singleton. Connections of this set with the notion of \textit{depth} of $a$ has also been explored. We show that $S_K(a)$ is a singleton whenever $K(a)|K$ is a minimal extension. A stronger version of this result is obtained when $a$ has depth one over $K$. We also provide a host of examples illustrating that the bounds obtained are strict.  
		
		\pars Apart from being of independent interest, another primary motivation for considering this problem comes from the study of ramification ideals. In the depth one case, when $K(a)|K$ is a Galois extension, we obtain intimate connections between the cardinalities of $S_K(a)$ and the number of ramification ideals of the extension $(K(a)|K,v)$. In particular, we show that these cardinalities are same whenever the extension is defectless and non-tame, or whenever $(K,v)$ has rank one. In order to obtain these results, we provide comprehensive descriptions of the ramification ideals of $(K(a)|K,v)$ which extend the known results in this direction.     
	\end{abstract}

	\section{Introduction}
	
	Let $(K,v)$ be a henselian non-trivially valued field and $\overline{K}$ denote a fixed algebraic closure of $K$. We denote the unique extension of $v$ to $\overline{K}$ again by $v$. The value group and residue field of $K$ are denoted respectively by $vK$ and $Kv$. The value and residue of an element $a$ is denoted by $va$ and $av$ respectively. 
	
	\pars Take some $a\in K^\sep \setminus K$, where $K^\sep$ denotes the separable-algebraic closure of $K$. We define 
	\[ S_K(a) := \{ v(a-a^\prime) \mid a^\prime\neq a \text{ is a $K$-conjugate of }a \}.  \] 
	Thus $1\leq \# S_K(a) \leq \deg_K(a)-1$. The primary motivation of this paper is to investigate the following problem proposed by Franz-Viktor Kuhlmann:
	
	\begin{Question}\label{Qn characterize S_K(a)=1}
		Give a characterization of all $a\in K^\sep\setminus K$ such that $\# S_K(a) = 1$.
	\end{Question}
	
	An initial observation is that the value $\#S_K(a)$ is a property of the element $a$ and not of the field extension $K(a)|K$. Indeed, there are field extensions $L|K$ with two distinct generators $a$ and $b$ such that $\#S_K(a) \neq \#S_K(b)$ (cf. Example \ref{Eg S_K(a) neq S_K(b)}). Nevertheless, we do have a complete answer to the problem in the setup of \emph{minimal} extensions. A field extension $\O|K$ is said to be minimal if there are no subextensions $F|K$ with $K\neq F \neq \O$. 
	
	\begin{Theorem}(Theorem \ref{Thm S_K(a)=1 minimal extn})
			Assume that $K(a)|K$ is a minimal extension. Then $\# S_K(a) = 1$.
	\end{Theorem}
	In particular, this shows that $\#S_K(a) = 1$ whenever $K(a)|K$ is of prime degree, a result which generalizes the known result for Artin-Schreier or Kummer generators. 
	
	Since $\#S_K(a)$ depends on the element $a$ rather than the field generated by it, we take a closer look at the properties of the element. Associated to each element is a positive integer called its \textit{depth} (cf. Section \ref{Sect Prelims} for the definition), denoted by $\ell(a)$. This integer appears in various forms throughout the literature; as lengths of complete distinguished chains (provided such a chain exists), lengths of Okutsu sequences, and as lengths of Mac Lane-Vaqui\'{e} chains. An element $a$ is said to be \textit{pure over $K$} whenever $\ell(a)=1$. In this direction, Kuhlmann proposed the following problem:
	
	\begin{Question}\label{Qn pure implies S_K(a)=1}
		Assume that $a$ is pure over $K$. Do we have $\# S_K(a) = 1$? 
	\end{Question}
	 
	 The answer to Question \ref{Qn pure implies S_K(a)=1} is affirmative whenever $(K(a)|K,v)$ is a tame extension (cf. Definition \ref{Defn tame} for the definition of tame extensions, and Theorem \ref{Thm tame S_K(a)} for a proof of the assertion), but is negative in the general setting (cf. Example \ref{Eg bound sharp a pure}). In the case of pure elements, a bound for the cardinality of the set $S_K(a)$ is provided by the next theorem: 
	 
	 \begin{Theorem}(Theorem \ref{Thm S_K(a) bound a pure})\label{Thm Intro S_K(a) bound a pure}
	 	Assume that $a$ is pure over $K$ and that $(K(a)|K,v)$ is not a tame extension. Then,
	 	\[ \#S_K(a) \leq \begin{cases*}
	 		v_p \deg_L(a) +1 \text{ whenever }\max v(a-K) \in S_K(a),\\
	 		v_p \deg_L(a) \text{ otherwise},
	 	\end{cases*}  \]
	 	where $v(a-K):= \{v(a-z)\mid z\in K\}$, $p$ is the characteristic exponent, $v_p$ denotes the $p$-adic valuation, $L:= K(a)\sect K^r$ and $K^r$ is the absolute ramification field of $(K,v)$.
	 \end{Theorem} 
     As an immediate corollary, we obtain a sufficient \textit{ramification theoretic} condition which guarantees an affirmative answer to Question \ref{Qn pure implies S_K(a)=1}:
     \begin{Corollary}(Corollary \ref{Coro S_K(a)=1 a pure K(a)|L minimal})\label{Coro Intro}
     	Let notations and assumptions be as in Theorem \ref{Thm Intro S_K(a) bound a pure}. Moreover, assume that $K(a)|L$ is a minimal extension. Then,
     		\[ \#S_K(a) = \begin{cases*}
     		2 \text{ whenever } \max v(a-K) \in S_K(a),\\
     		1 \text{ otherwise}.
     	\end{cases*}  \]
     \end{Corollary}
     The bounds obtained in Theorem \ref{Thm Intro S_K(a) bound a pure} and Corollary \ref{Coro Intro} are sharp (cf. Example \ref{Eg bound sharp a pure}). However, the inequality in the statement of Theorem \ref{Thm Intro S_K(a) bound a pure} can be strict, as observed in Example \ref{Eg strict inequality in Thm S_K(a) bound pure}. This inequality is obtained in the following way: we write $S_K(a) = \{ \g_0, \g_1, \dotsc , \g_m \}$ as a strictly decreasing sequence. To each $\g_i$, we associate an extension $w_i$ of $v$ to $K(X)$. We then construct a nested tower of field extensions 
     \[ L\subseteq L_{m-1} \subsetneq \dotsc \subsetneq L_0 \subsetneq K(a),   \]
     from which we obtain the desired bound. These fields appear as the implicit constant fields (introduced by Kuhlmann in [\ref{Kuh value groups residue fields rational fn fields}]) of the extensions $(K(X)|K,w_i)$. Employing the notion of the $j$-invariant (introduced by the first author in [\ref{Dutta min fields implicit const fields}] to study implicit constant fields), we exhibit that they form a nested chain. We refer the reader to Section \ref{Sect Prelims} for the relevant definitions and properties. 
     
	\pars In case of elements of higher depth, we have a generalization of Theorem \ref{Thm Intro S_K(a) bound a pure} (cf. Theorem \ref{Thm S_K(a) bound general}), under the assumption that $(K(a)|K,v)$ is a \textit{defectless} extension, that is when 
	\[ (vK(a):vK)[K(a)v:Kv] = \deg_K(a).  \]
	We are not aware of an extension of the result to the defect case. 
	
	\pars The primary motivation for considering Questions \ref{Qn characterize S_K(a)=1} and \ref{Qn pure implies S_K(a)=1} comes from the study of \textit{ramification ideals}. We refer the reader to Section \ref{Sect Ramification ideals} for the relevant definitions and properties, and to [\ref{Cutkosky-FVK-AR Computation of Kahler differentials ind defect}], [\ref{FVK AR Valn theory deeply ramified}] and [\ref{FVK AR Topics in higher ramification theory}] for motivations and background to studying these objects. For a given Galois extension $\mathcal{E}:= (\O|K,v)$, we denote by $\Ram(\mathcal{E})$ the set of ramification ideals of this extension. This set is non-empty only when $\mathcal{E}$ is not tame. Moreover, the set $\Ram(\mathcal{E})$ depends only on the extension $\mathcal{E}$ and is independent of the choice of the generator. We define the notion of \textit{depth of an extension} as follows: the depth of an extension $\mathcal{E} = (\O|K,v)$ is defined as
	\[ \depth (\mathcal{E}) := \min \{ \ell(a) \mid a \text{ is a generator of } \O|K \}. \]  
	The following theorem, proved over the course of Theorem \ref{Thm Ram(E) = S_K(a)}, Theorem \ref{Thm Ram(E) leq S_K(a) defect} and Theorem \ref{Thm Ram(E) = S_K(a) defect rank one}, provides an intimate connection of the set $S_K(a)$ with $\Ram(\mathcal{E})$:
	
	\begin{Theorem}
		Let $\mathcal{E}:= (\O|K,v)$ be a Galois extension of henselian valued fields. Assume that $\depth(\mathcal{E}) = 1$. Further assume that $\mathcal{E}$ is not tame. Then for any generator $a$ of $\O|K$ such that $a$ is pure over $K$, we have that
		\[ \#\Ram(\mathcal{E}) \leq \# S_L(a), \text{ where } L:= \O\sect K^r.  \]
		In particular, if $\mathcal{E}$ is purely wild, then we have that
		\[  \#\Ram(\mathcal{E}) \leq \# S_K(a).  \]
		Equality holds whenever $\mathcal{E}$ is defectless or when $\rk (K,v) = 1$.
	\end{Theorem}
	
	The problem remains open whether the equality holds true for every pure generator of a depth one extension with defect for valued fields of higher rank. Moreover, even in the defectless setup, the problem remains open whether the assertions can be extended to Galois extensions of higher depth.


	\section{Preliminaries}\label{Sect Prelims}
	
	\subsection{The defect} The Lemma of Ostrowski states that for a finite extension $(L|K,v)$, we always have that
	\[  [L:K] = (vL:vK)[Lv:Kv]p^n,   \]
	where $p$ is the \textbf{characteristic exponent}, that is, $p = 1$ if $\ch Kv = 0$ and $p=\ch Kv$ otherwise. The factor $p^n$ is said to be the defect of the extension $(L|K,v)$ and is denoted by $d(L|K,v)$. We say that $(L|K,v)$ is \textbf{defectless} if $d(L|K,v) = 1$. We refer the reader to [\ref{Kuh defect}] for an extensive treatment on the defect. 
	

	\subsection{Complete distinguished chains} A characterization of defectless simple extensions over \textit{henselian} valued fields has been provided in [\ref{KA SKK chains associated with elts henselian}] via the notion of complete distinguished chains. A pair $(b,a) \in \overline{K}\times\overline{K}$ is said to form a \textbf{distinguished pair over $K$} if the following conditions are satisfied:
	\sn (DP1) $\deg_K(b) > \deg_K(a)$,
	\n (DP2) $\deg_K(z) < \deg_K(b) \Longrightarrow v(b-a)\geq v(b-z)$, 
	\n (DP3) $v(b-a) = v(b-z)\Longrightarrow \deg_K(z) \geq \deg_K(a)$. \\
	In other words, $a$ is closest to $b$ among all the elements $z$ satisfying $\deg_K(z) < \deg_K(b)$; furthermore, $a$ has minimum degree among all such elements which are closest to $b$. In this case we define $\d_K(b):= v(b-a)$. Equivalently,
	\[ \d_K(b):= \max \{ v(b-z) \mid \deg_K(z) < \deg_K(b) \}.  \]
	An element $a\in \overline{K}$ is said to admit a \textbf{complete distinguished chain over $K$} if there is a chain $a_0(= a), a_1, \dotsc , a_n$ of elements in $\overline{K}$ such that $(a_i, a_{i+1})$ is a distinguished pair over $K$ for all $i$, and $a_n \in K$. Observe that, 
	\[ \d_K(a) > \d_K(a_1) > \dotsc > \d_K(a_{n-1}) = v(a_{n-1} - a_n).  \]
	In the setup of henselian valued fields, the existence of complete distinguished chains is provided by [\ref{KA SKK chains associated with elts henselian}, Theorem 1.2]:
	\begin{Theorem}\label{Thm cdc Khanduja-Aghigh}
		Let $(K,v)$ be a henselian valued field. Then an element $a\in \overline{K}$ admits a complete distinguished chain over $K$ if and only if $(K(a)|K,v)$ is a defectless extension. 
	\end{Theorem} 
A generalization of this theorem in the setup of \textit{unibranched} simple defectless extensions over \textit{arbitrary} valued fields has been provided in [\ref{Dutta Ghosh defectless unibranched extn complete distinguished chains}, Theorem 1.2].

	
	\subsection{Okutsu sequences} If $(K(a)|K,v)$ is an arbitrary (not necessarily defectless) extension of henselian valued fields, then we can replace the notion of complete distinguished chains by that of \textit{Okutsu sequences}. We first introduce a couple of notations. 
	
	\pars If $\G$ is a subset of $vK$ and $\g\in vK$, then we say $\g \geq \G$ if $\g \geq \a$ for all $\a\in\G$. Thus either $\g = \max \G$, or else $\g > \G$. 
 	
 	\pars If $A \subseteq \overline{K}$, we say that $A$ has a \textit{common degree} if all its elements have the same degree over $K$. In this case, we denote the common degree by $\deg_K(A)$. 
 	
 	\pars For any positive integer $d < \deg_K(a)$, we define 
 	\[ D_d  = D_d(a,K):= \{ v(z-a) \mid z\in \overline{K}, \, \deg_K(z) = d \}.    \]
 	The set $D_1(a,K)$ is commonly expressed as $v(a-K)$ in the literature. 
 	
 	\begin{Definition}
 		An \textbf{Okutsu sequence} of $a$ over $K$ is a finite sequence 
 		\[  A_0 (= \{a\}), A_1, \dotsc , A_n,  \]
 		of common degree subsets of $\overline{K}$ satisfying the following conditions:
 		\sn (OS1) $\deg_K(a) = d_0 > d_1 > \dotsc > d_n = 1$, where $d_i:= \deg_K(A_i)$.
 		\n (OS2) For all $z\in\overline{K}$ with $\deg_K(z) < d_{i-1}$, there exists some $z_i\in A_i$ such that $v(a-z_i) \geq v(a-z)$.
 		\n (OS3) For all $z_i \in A_i$ and $z_{i-1} \in A_{i-1}$, we have that $v(a-z_{i-1}) > v(a-z_i)$. Thus $v(a-z_{i-1}) > D_{d_i}$.
 		\n (OS4) $\# A_i = 1$ if and only if $\max D_{d_i}$ exists. In this case, $A_i = \{a_i\}$ where 
 		\[  v(a-a_i) = \max D_{d_i} = \max \{ v(a-z) \mid z\in \overline{K}, \, \deg_K(z) < d_{i-1} \}.  \]
 		\n (OS5) If $\max D_{d_i}$ does not exist, then we set $A_i= \{z_\nu\}_{\nu<\l}$ where $\l$ is a limit ordinal, such that 
 		\[ v(a-z_\mu) > v(a-z_\nu) \text{ for all ordinals } \nu < \mu < \l.  \]
 	\end{Definition}
	We say that the above Okutsu sequences has length $n$.
	
	\begin{Remark}
		If $(K(a)|K,v)$ is a defectless extension, then $a,a_1, \dotsc , a_n$ is a complete distinguished chain of $a$ over $K$ if and only if $\{a\}, \{a_1\}, \dotsc , \{a_n\}$ is an Okutsu sequence of $a$ over $K$. In this case, 
		\[ v(a-a_i) = \d_K(a_{i-1}) = \max D_{d_i} \text{ for all } i\geq 1.  \]
	\end{Remark}
	

\subsection{Pseudo convergent sequences} 

\begin{Definition}
	A well-ordered set $\{z_\nu\}_{\nu<\l} \subset K$, where $\l$ is a limit ordinal, is said to be a \textbf{pseudo convergent sequences} (pcs) if 
	\[  v(z_\nu - z_\mu) < v(z_\mu - z_\rho) \text{ whenever } \nu<\mu<\rho<\l.  \]
\end{Definition}

\begin{Definition}
	Let $(\O|K,v)$ be an extension of valued fields and $\{z_\nu\}_{\nu<\l}$ a pcs in $K$. An element $x\in\O$ is said to be a \textbf{limit} of $\{z_\nu\}_{\nu<\l}$ if 
	\[ v(x-z_\nu) = v(z_\nu - z_{\nu +1}) \text{ for all } \nu<\l.   \]
\end{Definition}

Pseudo convergent sequences are fundamental to the study of \textit{immediate} extensions, i.e. extensions which render the value group and residue field unchanged. We refer the reader to [\ref{Kaplansky}] for a thorough treatment of these objects. 

\begin{Lemma}[\ref{Kaplansky}, Lemma 1 and Lemma 5]\label{Lemma {f(z_nu)} ultimately pcs}
	Let $\{z_\nu\}_{\nu<\l}$ be a pcs in $K$. Take $f(X)\in K[X]$. Then $\{vf(z_\nu)\}_{\nu<\l}$ is either ultimately stable or ultimately strictly increasing. 
\end{Lemma}

\begin{Lemma}[\ref{Dutta min fields implicit const fields}, Lemma 9.1]\label{Lemma Dutta pcs}
	Let $\mathbf{z}:= \{z_\nu\}_{\nu<\l}$ be a pcs in $K$ without a limit in $K$. Then the following statements hold true:
	\sn (i) Assume that $f(X)\in K[X]$ is such that $\{vf(z_\nu)\}_{\nu<\l}$ is ultimately strictly increasing. Then at least one root of $f$ is a limit of $\mathbf{z}$.
	\n (ii) Assume that $a\in\overline{K}$ be a limit of $\mathbf{z}$. Take the minimal polynomial $Q(X)$ of $a$ over $K$. Then $\{vQ(z_\nu)\}_{\nu<\l}$ is ultimately strictly increasing. 
\end{Lemma}


\subsection{Depth and Purity} Take any $a\in\overline{K}\setminus K$. Then $a$ admits an Okutsu sequence over $K$, as observed in [\ref{Nart Okutsu sequences henselian}, \ref{Nart Novacoski Depth of extns of valns}]. The following theorem has been proved in the language of Mac Lane-Vaqui\'{e} chains in the henselian setup in [\ref{Nart Okutsu sequences henselian}], and in [\ref{Nart Novacoski Depth of extns of valns}, Theorem 1.4] in the general case. In the language of complete distinguished chains, the result appears in [\ref{KA SKK chains associated with elts henselian}]. 

\begin{Theorem}
	Every Okutsu sequence of $a$ over $K$ has the same length. 
\end{Theorem}
We call this length the \textbf{depth of $a$ over $K$} and denote it by $\ell_K(a)$. When the underlying field is tacitly understood, we simply denote it by $\ell(a)$.

\begin{Definition}
	We say that $a$ is \textbf{pure over $K$} if $\ell(a) = 1$.
\end{Definition}

\begin{Remark}\label{Rmk characterization pure elts}
	Assume that $a$ is pure over $K$. Take an Okutsu sequence $A_0 = \{a\}, A_1$. Note that $d_1 = \deg_K(A_1) = 1$. Then the following cases are possible: 
	\sn (i) $d(K(a)|K,v) = 1$. In this case, $A_1 = \{a_1\}$ and $(a,a_1)$ is a distinguished pair over $K$ as well as a complete distinguished chain of $a$ over $K$. In particular, 
	\[ v(a-a_1) = \max v(a-K).  \] 
	\n (ii) $d(K(a)|K,v) > 1$. In this case, $A_1 = \{z_\nu\}_{\nu<\l} \subset K$ is a pcs in $K$. By definition, for all $b\in\overline{K}$ with $\deg_K(b) < \deg_K(a)$, we have some $\nu<\l$ such that $v(a-b)\leq v(a-z_\nu)$. It follows that $v(a-K)$ has no maximal element. Moreover, $a$ is a limit of $A_1$, and has least degree over $K$ among all its algebraic limits. As a consequence, we obtain the following from Lemma \ref{Lemma Dutta pcs} and [\ref{Kaplansky}, Theorem 3]:
	\sn (A) $(K(a)|K,v)$ is immediate, that is, $vK(a) = vK$ and $K(a)v = Kv$,
	\n (B) for every $g(X)\in K[X]$ with $\deg g < \deg_K(a)$, there exists some $\nu_0 < \l$ such that
	\[ vg(z_\nu) = vg(a) \text{ for all } \nu_0 \leq \nu < \l.  \]
\end{Remark}

	
	\subsection{The Krasner's constant} 
	
	\begin{Definition}
		The \textbf{Krasner's constant} is denoted by $\o_K(a)$ and defined as
		\[ \o_K (a) := \max S_K(a).  \]
	\end{Definition}
	Thus,
	\[ \# S_K(a) = 1 \text{ if and only if } v(a-a^\prime) = \o_K(a) \text{ for all $K$-conjugates $a^\prime$ of $a$}.  \]
	
	\pars We state a variant of the important \textbf{Krasner's Lemma} [\ref{Kuh value groups residue fields rational fn fields}, Lemma 2.21]:
	
	\begin{Lemma}\label{Lemma Krasner}
		Take $a\in K^\sep\setminus K$ and $b\in \overline{K}$. Then
		\[  v(a-b) > \o_K(a) \Longrightarrow K(a) \subseteq K(b).  \]
	\end{Lemma}
As a consequence, we obtain that
\begin{equation}\label{Eqn 1}
	D_{d_1} \leq \o_K(a) \text{ whenever } A_0, A_1, \dotsc , A_n \text{ is an Okutsu sequence of $a$ over $K$}.
\end{equation}


\subsection{Some ramification theory} Set $G:= \Gal(\overline{K}|K)$. The fact that $(K,v)$ is henselian implies that $v \s a = va$ for all $a\in \overline{K}$ and $\s\in G$. We define the \textbf{absolute ramification group} of $(K,v)$ as 
\[ G^r := \{ \s\in G \mid v(\s a - a) > va \text{ for all } a\in K^\sep\setminus \{0\}  \}.  \]
The corresponding fixed field in $K^\sep$ is denoted by $K^r$ and is called the \textbf{absolute ramification field} of $(K,v)$. For an arbitrary algebraic extension $(L|K,v)$, we have that $L^r = L.K^r$ [\ref{Kuh vln model}, Theorem 5.10], where $L.K^r$ denotes the compositum of $L$ and $K^r$. 

\begin{Definition}\label{Defn tame}
	An algebraic extension $(L|K,v)$ is said to be \textbf{tame} if every finite subextension $(E|K,v)$ satisfies the following conditions:
	\sn (TE1) $(vE:vK)$ is coprime to $p$,
	\n (TE2) $Ev|Kv$ is separable, 
	\n (TE3) $d(E|K,v) = 1$. 
\end{Definition}
It is well-known that $(L|K,v)$ is tame if and only if $L\subseteq K^r$ [\ref{Kuh vln model}, Theorem 11.1].

\begin{Lemma}\label{Lemma va = kras(a) tame}
	Let $(K(a)|K,v)$ be a separable extension. Assume that $v(a-d) = \o_K(a)$ for some $d\in K$. Then $(K(a)|K,v)$ is a tame extension.
\end{Lemma}

\begin{proof}
	The fact that $d\in K$ implies that $\o_K(a) = \o_K(a-d)$. We can thus assume that $va = \o_K(a)$ without any loss of generality. It is enough to show that $a\in K^r$. Suppose the contrary. Then there exists $\s\in G^r$ such that $\s a \neq a$. In other words, $\s a \neq a$ and $v(\s a - a) > va = \o_K(a)$ which thus yields a contradiction.
\end{proof}

\begin{Definition}
	An algebraic extension $(L|K,v)$ is said to be \textbf{purely wild} if $vL/vK$ is a $p$-group and $Lv|Kv$ is purely inseparable. Equivalently, $L\sect K^r = K$.
\end{Definition}  

\pars We now state the following lemma which will be required later.

\begin{Lemma}[\ref{Dutta non-uniqueness max purely wild extns}, Lemma 2.5]\label{Lemma tame purely wild valn disjoint}
	Let $(L|K,v)$ be a tame extension and $(F|K,v)$ be purely wild. Then the following statements hold true:
	\sn (i) $L$ and $F$ are linearly disjoint over $K$,
	\n (ii) $v(L.F) = vL + vF$ and $(L.F)v = Lv. Fv$,
	\n (iii) $vL\sect vF = vK$ and the residue fields $Lv$ and $Fv$ are linearly disjoint over $Kv$,
	\n (iv) $(L.F|F,v)$ is a tame extension and $(L.F|L,v)$ is purely wild.   
\end{Lemma}


\subsection{(Minimal) Pairs of definition} Let $K(X)$ be a rational function field over $K$. Take an extension $w$ of $v$ to $K(X)$ and assume that it admits equality in the \textbf{Abhyankar Inequality}, that is, 
\[ \dim_{\QQ} (\QQ \tensor_{\ZZ} wK(X)/vK) + \trdeg [K(X)w:Kv] = 1.  \]
In this case we say that $w$ is a \textbf{valuation transcendental extension} of $v$ to $K(X)$. We fix an extension of $w$ to $\overline{K}(X)$ and denote it again by $w$. It is well-known (cf. [\ref{AP sur une classe}, \ref{APZ characterization of residual trans extns}, \ref{APZ2 minimal pairs}, \ref{Dutta min fields implicit const fields}, \ref{Kuh value groups residue fields rational fn fields}]) that $w$ is completely characterised by a pair $(a,\g) \in \overline{K}\times w\overline{K}(X)$ in the following sense:
\[ w(X-z) = \min\{ \g, v(a-z) \} \text{ for all } z\in \overline{K}.  \]
We say that $(a,\g)$ is a \textbf{pair of definition for $w$} and write $w= v_{a,\g}$. A pair of definition for $w$ may not be unique. It has been observed in [\ref{AP sur une classe}, Proposition 3] that
\[ (b,\g^\prime) \text{ is also a pair of definition for $w$ if and only if } v(a-b)\geq \g=\g^\prime.  \] 

\begin{Definition}
	We say that $(a,\g)$ is a \textbf{minimal pair of definition for $w$ over $K$} if it has minimal degree over $K$ among all pairs of definition, that is, 
	\[  v(a-b)\geq \g \Longrightarrow \deg_K(b)\geq \deg_K(a). \]
\end{Definition}

The following fact appears in [\ref{Dutta imp const fields key pols valn alg extns}, Lemma 3.2]:

\begin{Lemma}\label{Lemma vK(a) < vK(b) and K(a)v < K(b)v}
	Take a minimal pair of definition $(a,\g)$ for $w$ over $K$ and a pair of definition $(b,\g)$. Then $vK(a)\subseteq vK(b)$ and $K(a)v\subseteq K(b)v$. 
\end{Lemma}

We will also require the following easy observation:

\begin{Lemma}\label{Lemma min pair of defn over K(z)}
	Take a minimal pair of definition $(a,\g)$ for $w$ over $K$. Take some $z\in K(a)$. Then $(a,\g)$ is a minimal pair of definition for $w$ over $K(z)$ as well.
\end{Lemma}

\begin{proof}
	Take some minimal pair of definition $(b,\g)$ for $w$ over $K(z)$. Hence $\deg_{K(z)}(b) \leq \deg_{K(z)}(a)$. Since $z\in K(a)$, it follows that $[K(b,z):K] \leq \deg_K(a)$. As a consequence, 
	\[ \deg_K(b) \leq  [K(b,z):K] \leq \deg_K(a) \leq \deg_K(b), \]
	where the last inequality follows from the minimality of $(a,\g)$. The assertion now follows.
\end{proof}

The connection between minimal pairs of definition and distinguished pairs is captured by the following observation which is immediate and hence its proof is omitted:
\begin{Proposition}\label{Prop dist pair min pair}
	Assume that $(b,a)$ is a distinguished pair over $K$ and set $\g = \d(b,K)$. Then $(a,\g)$ is a minimal pair of definition for $v_{b,\g}$ over $K$. Moreover, if $\g^\prime$ is an element in an ordered abelian group containing $v\overline{K}$ such that $\g^\prime>\g$, then $(b,\g^\prime)$ is a minimal pair of definition for $v_{b,\g^\prime}$ over $K$.  
\end{Proposition} 	
	

\subsection{$j$-invariant and Implicit Constant Field} Fix an extension of $w$ to $\overline{K(X)}$, denoted again by $w$. 

\begin{Definition}
	The \textbf{Implicit Constant Field} of the extension $(K(X)|K,w)$ is defined as
	\[ IC_K(w) := \overline{K} \sect K(X)^h,  \]
	where $K(X)^h$ denotes the \emph{henselization} of $(K(X),w)$. 
\end{Definition}

Observe that $IC_K(w)$ is an algebraic extension of $K$ and hence is itself henselian. Moreover, it follows from the definition that
\[ IC_K(w) \subseteq IC_L(w) \text{ whenever }  K\subseteq L.  \]

\begin{Remark}
	The notion of Implicit Constant Fields was introduced by Kuhlmann in [\ref{Kuh value groups residue fields rational fn fields}] to construct extensions with prescribed value groups and residue fields. They also play a pivotal role in studying extensions over tame fields (cf. [\ref{Dutta FVK tamekeypols}]).
\end{Remark}

The following theorem provides bounds for the Implicit Constant Field. For a complete proof, we refer the reader to [\ref{Dutta min fields implicit const fields}, Theorem 1.3] and [\ref{Dutta min pairs inertia ram deg impl const fields}, Theorem 1.3]. 

\begin{Theorem}\label{Thm bound of ICF}
	Take a minimal pair of definition $(a,\g)$ for $w$ over $K$. Set $L:= K(a)\sect K^r$. Then,
	\[ L \subseteq IC_K(w) \subseteq K(a).  \]
\end{Theorem}

\begin{Remark}
	The upper bound of $IC_K(w)$ in Theorem \ref{Thm bound of ICF} holds true for any arbitrary pair of definition (cf. [\ref{Dutta min fields implicit const fields}, Lemma 5.1]). However, the minimality is crucial in obtaining the lower bound, a fact which is essential in Proposition \ref{Prop S_K(a) = 1 necessary condns} and Theorem \ref{Thm S_K(a) bound a pure}. 
\end{Remark}

\begin{Remark}
	Observe that $L^r = L.K^r = K^r$ and hence $(K(a)|L,v)$ is purely wild. It follows that the subextensions $(K(a)|IC_K(w),v)$ and $(IC_K(w)|L,v)$ are purely wild extensions. 
\end{Remark}

\pars We associate an invariant to $w$, called the \textbf{$j$-invariant}, in the following way.

\begin{Definition}
	Take a pair of definition $(a,\g)$ for $w$. Take a polynomial $f(X)\in \overline{K}[X]$ and consider its decomposition $f(X) = (X-z_1)\dotsc (X-z_n), \, z_i \in \overline{K}$. We define 
	\[ j_w(f):= \# \{z_i \mid v(a-z_i)\geq \g\}.   \] 
	When the valuation $w$ is tacitly understood, we will drop the suffix and simply write it as $j(f)$. 
\end{Definition}

The $j$-invariant was introduced by the first author in [\ref{Dutta min pairs inertia ram deg impl const fields}] and later studied in detail in [\ref{Dutta invariant of valn tr extns and connection with key pols}] and [\ref{Dutta Ghosh defectless unibranched extn complete distinguished chains}]. It provides a measure of how far the Implicit Constant Field is from attaining the upper bound in Theorem \ref{Thm bound of ICF}.  

\begin{Theorem}[\ref{Dutta min pairs inertia ram deg impl const fields}, Theorem 1.1 and \ref{Dutta Ghosh defectless unibranched extn complete distinguished chains}, Proposition 2.5] \label{Thm j = [K(a):ICF] min pair of defn} 
	Take a minimal pair of definition $(a,\g)$ for $w$ over $K$ and take the minimal polynomial $Q(X)$ of $a$ over $K$. Then $d(K(a,X)^h|K(X)^h, w) = 1$. Moreover, 
	\[ j(Q) = [K(a):IC_K(w)].   \]
\end{Theorem}

Under certain specific conditions, we can have the same conclusions even if we forego minimality of the pair of definition in Theorem \ref{Thm j = [K(a):ICF] min pair of defn}. In this direction, we mention the following result which follows from [\ref{Dutta Ghosh defectless unibranched extn complete distinguished chains}, Lemma 3.1, Corollary 2.8 and Theorem 1.5]:

\begin{Theorem}\label{Thm j = [K(a):ICF] key pol}
	Assume that $(a,a_1)$ is a distinguished pair. Take some $\g\geq \d_K(a)$ and set $w:= v_{a,\g}$. Take the minimal polynomial $Q(X)$ of $a$ over $K$. Then $d(K(a,X)^h|K(X)^h,w) = 1$ and 
	\[ j(Q) = [K(a):IC_K(w)].  \]
\end{Theorem} 


\section{Notations}\label{Sect Notations}

We now fix some notations which will be used throughout Sections 4--8. We assume that $(K(a)|K,v)$ is a non-trivial separable extension of henselian valued fields. Fix an Okutsu sequence $A_0 = \{a\},A_1, \dotsc , A_n$ of $a$ over $K$ and set $d_i:= \deg_K(A_i)$ for all $i$. Moreover, whenever $\#A_i = 1$, we take $a_i$ such that $A_i = \{a_i\}$. Thus if $(K(a)|K,v)$ is defectless, we have a complete distinguished chain $a,a_1, \dotsc , a_n$ of $a$ over $K$. Observe that
\[ D_{d_1} = v(a-K) \text{ whenever $a$ is pure over $K$}.   \] 

\pars Take the minimal polynomial $Q(X)$ of $a$ over $K$. Set 
\[ \text{$\g:= \o_K(a)$ and $w:= v_{a,\g}$.}    \]
Fix an extension of $w$ to $\overline{K(X)}$. It has been observed in (\ref{Eqn 1}) that $D_{d_1} \leq \g$. As a consequence, it follows from Proposition \ref{Prop dist pair min pair} that $(a,\g)$ is a minimal pair of definition for $w$ over $K$ whenever $\g > D_{d_1}$. Otherwise, $\g = \max D_{d_1} = v(a-a_1)$. In this case $(a,a_1)$ is a distinguished pair over $K$ and $(a_1,\g)$ is a minimal pair of definition for $w$ over $K$. Furthermore, observe that
\[ \# S_K(a) = 1 \text{ if and only if } j(Q) = \deg Q.  \]
In light of Theorem \ref{Thm j = [K(a):ICF] key pol}, we conclude that
\begin{equation}\label{Eqn 2}
	\# S_K(a) = 1 \text{ if and only if } IC_K(w) = K. 
\end{equation}

We combine the preceding observations in the next result:

\begin{Theorem}\label{Thm S_K(a) = 1 iff condns} 
	We have that
	\[ \#S_K(a) = 1 \Longleftrightarrow j(Q) = \deg Q \Longleftrightarrow IC_K(w) = K.  \]
\end{Theorem}


\section{The tame case} 

When $(K(a)|K,v)$ is a tame extension, then Questions \ref{Qn characterize S_K(a)=1} and \ref{Qn pure implies S_K(a)=1} have very satisfactory answers.

\begin{Theorem}\label{Thm tame S_K(a)}
	$(K(a)|K,v)$ is a tame extension if and only if $\d_K(a_i) = \o_K(a_i)$ for all $i$. 
	\newline In this case, $K(a_i) \supsetneq K(a_{i+1})$ for all $i$. Moreover, $S_K(a) = \{ \d_K(a), \d_K(a_1), \dotsc , \d_K(a_{n-1}) \}$ and hence $\#S_K(a) = \ell(a)$.
\end{Theorem} 

\begin{proof}
	We first assume that $(K(a)|K,v)$ is a tame extension. Suppose that $\d_K(a) < \o_K(a) = \g$. Then $(a,\g)$ is a minimal pair of definition for $w$ over $K$. By Theorem \ref{Thm bound of ICF} we have that $IC_K(w) = K(a)$ and hence $j(Q) = 1$ by Theorem \ref{Thm j = [K(a):ICF] min pair of defn}. However this is not possible since $j(Q)>1$ by the definition of $\g$. It follows that $\d_K(a) = \g$. Now observe from Lemma \ref{Lemma vK(a) < vK(b) and K(a)v < K(b)v} and Theorem \ref{Thm cdc Khanduja-Aghigh} that $(K(a_i)|K,v)$ is a tame extension for all $i$. The same arguments now yield that $\d_K(a_i) = \o_K(a_i)$ for all $i>0$.
	
	\pars We will prove the reverse direction by induction on $\ell(a)$. When $\ell(a)=1$, the assertion follows from Lemma \ref{Lemma va = kras(a) tame}. Now assume that $\ell(a) = n > 1$, $\d_K(a_i) = \o_K(a_i)$ for all $i$, and that the assertion holds true for all elements $b$ with $\ell(b) \leq n-1$ and satisfying the given conditions. Thus $(K(a_1)|K,v)$ is tame by the induction hypothesis. Observe that
	\[ \o_{K(a_1)}(a) \geq \d_{K(a_1)}(a) \geq v(a-a_1) = \o_K(a) \geq \o_{K(a_1)}(a).   \]
	As a consequence, $v(a-a_1) = \o_{K(a_1)}(a)$ and hence $(K(a,a_1)|K(a_1),v)$ is a tame extension by Lemma \ref{Lemma va = kras(a) tame}. The fact that $(K(a_1)|K,v)$ is a tame extension implies that $a_1 \in K^r$. It follows that $a\in K(a_1)^r = K^r.K(a_1) = K^r$ and hence $(K(a)|K,v)$ is a tame extension. 
	
	\pars The second assertion follows from Lemma \ref{Lemma Krasner}. The final assertion has been mentioned in [\ref{Novacoski distances of elts to conjugates}, Theorem 3.1], and appears multiple times in the literature.     
\end{proof}

In the rest of the paper, we will tackle the complementary case.


\section{The condition $\# S_K(a) = 1$} We first provide some necessary conditions for the condition $\#S_K(a) = 1$ to hold. Recall that we are in the setup of Section \ref{Sect Notations}. 

\begin{Proposition}\label{Prop S_K(a) = 1 necessary condns}
	Assume that $\# S_K(a) = 1$. Set $L:= K(a)\sect K^r$. Then the following cases are possible:
	\sn (i) $D_{d_1} < \o_K(a)$. Then $(K(a)|K,v)$ is a purely wild extension. 
	\n (ii) $\max D_{d_1} = \o_K(a)$. Then $(K(a_1)|K,v)$ is a purely wild extension. Moreover,
	\[  \deg_K(a) = \deg_K(a_1)[L:K], \, vK(a) = vL + vK(a_1), \, K(a)v = Lv.K(a_1)v.   \]
\end{Proposition}

\begin{proof}
	We first assume that $D_{d_1} < \o_K(a) = \g$. Then $(a,\g)$ is a minimal pair of definition for $w$ over $K$. The first assertion now follows from Theorem \ref{Thm S_K(a) = 1 iff condns} and Theorem \ref{Thm bound of ICF}. 
	
	\pars We now assume that $\g = \max D_{d_1} = v(a-a_1)$. Then $(a_1, \g)$ is a minimal pair of definition for $w$ over $K$. The assertion that $(K(a_1)|K,v)$ is purely wild follows again from Theorems \ref{Thm S_K(a) = 1 iff condns} and \ref{Thm bound of ICF}. It remains to prove the final assertion. Employing Lemma \ref{Lemma tame purely wild valn disjoint} and Lemma \ref{Lemma vK(a) < vK(b) and K(a)v < K(b)v}, we obtain that
	\begin{align*}
		vL(a_1) &= vL + vK(a_1) \subseteq vK(a), \\
		L(a_1)v &= Lv.K(a_1)v \subseteq K(a)v.
	\end{align*}
It further follows from Lemma \ref{Lemma tame purely wild valn disjoint} and the multiplicative property of the defect that $d(L(a_1)|K,v) = d(K(a_1)|K,v)$. The Fundamental Inequality then gives us that 
\[ [L(a_1):K] = (vL(a_1):vK)[L(a_1)v:Kv]d(K(a_1)|K,v) \leq (vK(a):vK)[K(a)v:Kv]d(K(a_1)|K,v).  \]
Since $(a,a_1)$ is a distinguished pair, it has been observed in the proof of [\ref{KA SKK main invaraint of elts alg over henselian field}, Theorem 1.1] that $d(K(a)|K,v) = d(K(a_1)|K,v)$. As a consequence, 
\begin{equation}\label{Eqn 3}
	[L(a_1):K] \leq [K(a):K].
\end{equation}
Hence we are done if $a\in L(a_1)$. Now assume that $a\notin L(a_1)$. The fact that $\#S_K(a) = 1$ implies that $\#S_{L(a_1)}(a) = 1$ and thus $\o_K(a) = \o_{L(a_1)}(a)$. Since $v(a-a_1) = \o_K(a)$, we conclude from Lemma \ref{Lemma va = kras(a) tame} that 
\begin{equation}\label{Eqn 4}
	(L(a_1,a)|L(a_1),v) \text{ is a tame extnesion.}
\end{equation}
We now consider the following chain of value groups:
\[ vL \subseteq vL(a_1) \subseteq vK(a) \subseteq vL(a_1,a).  \]
Now $(vL(a_1,a):vL(a_1))$ is coprime to $p$ by (\ref{Eqn 4}). On the other hand, observe that $(K(a)|L,v)$ is purely wild by definition and hence $vK(a)/vL$ is a $p$-group. It follows that
\[ vK(a) = vL(a_1).  \] 
Analogously, considering the chain of residue field extensions
\[  Lv \subseteq L(a_1)v \subseteq K(a)v \subseteq L(a_1,a)v,  \]
we conclude that
\[ K(a)v = L(a_1)v.  \]
As a consequence, it follows from (\ref{Eqn 3}), the fact that $d(K(a)|K,v) = d(K(a_1)|K,v)$ and Lemma \ref{Lemma tame purely wild valn disjoint} that
\[ \deg_K(a) = \deg_K(a_1)[L:K].  \]
\end{proof}

We can now provide a complete answer to Question \ref{Qn characterize S_K(a)=1} in the setup of \textit{minimal} extensions. 

\begin{Definition}
	An extension of fields $L|K$ is said to be \textbf{minimal} if it does not admit any proper non-trivial subextension.
\end{Definition}

\begin{Theorem}\label{Thm S_K(a)=1 minimal extn}
	Assume that $K(a)|K$ is a minimal extension. Then $\# S_K(a) = 1$. 
\end{Theorem}

\begin{proof}
	Set $L:= K(a)\sect K^r$. The minimality of $K(a)|K$ implies that either $L = K(a)$ or $L=K$. In other words, either $(K(a)|K,v)$ is a tame extension, or it is purely wild. 
	
	\pars We first assume that $(K(a)|K,v)$ is a tame extension. Take a complete distinguished chain $a,a_1, \dotsc, a_n$ of $a$ over $K$. It follows from Theorem \ref{Thm tame S_K(a)} that $K(a_i) \subsetneq K(a)$ for all $i\geq 1$. The minimality of $K(a)|K$ then implies that $n=1$, i.e., $\ell(a) = 1$. Then $\# S_K(a) = 1$ by Theorem \ref{Thm tame S_K(a)}.
	
	\pars We now assume that $(K(a)|K,v)$ is purely wild. It follows from (\ref{Eqn 1}) that $D_{d_1} \leq \o_K(a) = \g$. Suppose if possible that $\g = \max D_{d_1}$. Then $A_1 = \{ a_1 \}$ where $(a,a_1)$ is a distinguished pair and $v(a-a_1) = \g$. Then $(a_1,\g)$ is a minimal pair of definition for $w:= v_{a,\g}$ over $K$. It follows from Theorem \ref{Thm j = [K(a):ICF] key pol} that $j(Q) = [K(a):IC_K(w)]$. Since $j(Q) > 1$ by the definition of $\g$, we conclude from the minimality of $K(a)|K$ that $IC_K(w) = K$. Thus $\# S_K(a) = 1$ by Theorem \ref{Thm S_K(a) = 1 iff condns}. Proposition \ref{Prop S_K(a) = 1 necessary condns} then yields that $\deg_K(a) = \deg_K(a_1)$ which is not possible. Thus we have $D_{d_1} < \g$. Hence $(a,\g)$ is a minimal pair of definition of $w$ over $K$. Thus $j(Q) = [K(a):IC_K(w)]$ by Theorem \ref{Thm j = [K(a):ICF] min pair of defn}. The same arguments again yield that $\# S_K(a) = 1$.
\end{proof}

\section{Analysis of pure elements}

\subsection{Initial results} A complete picture of pure elements is provided by the next proposition.

\begin{Proposition}\label{Prop pure elt implications}
	Assume that $a$ is pure over $K$. Then the following cases are possible:
	\sn (i) $\max v(a-K) = \o_K(a)$. Then $(K(a)|K,v)$ is tame and $\#S_K(a) = 1$.
	\n (ii) $v(a-K) < \o_K(a)$. Then there exists a proper subextension $F|K$ of $K(a)|K$ such that $\#S_F(a) = 1$.
\end{Proposition}

\begin{proof}
	The first assertion follows from Lemma \ref{Lemma va = kras(a) tame} and Theorem \ref{Thm tame S_K(a)}. We now assume that $v(a-K) < \g = \o_K(a)$. Then $(a,\g)$ is a minimal pair of definition for $w$ over $K$. Set $F:= IC_K(w)$. Then $[K(a):F] = j(Q) > 1$ by Theorem \ref{Thm j = [K(a):ICF] min pair of defn} and hence $F|K$ is a proper subextension of $K(a)|K$. Observe that $(a,\g)$ is also a minimal pair of definition for $w$ over $F$ by Lemma \ref{Lemma min pair of defn over K(z)}. Take the minimal polynomial $\tilde{Q}$ of $a$ over $F$. Then $j(\tilde{Q}) = [K(a):IC_F(w)]$ by Theorem \ref{Thm j = [K(a):ICF] min pair of defn}. By definition, $F\subseteq K(X)^h$. It now follows from [\ref{Kuh vln model}, Theorem 5.10] that $F(X)^h = F. K(X)^h = K(X)^h$. Consequently, $IC_F(w) = F$. Thus
	\begin{equation}\label{Eqn 5}
		j(\tilde{Q}) = \deg_F(a) = \deg\tilde{Q}.
	\end{equation}
	It follows that $\g \leq \o_F(a)$. Since $\o_F(a) \leq \o_K(a) = \g$, we conclude that $\o_F(a) = \g$. Hence $\#S_F(a) = 1$ by (\ref{Eqn 5}).
\end{proof}

\begin{Proposition}
	Consider the setup of Case (ii) of Proposition \ref{Prop pure elt implications}. Take any subextension $F^\prime|K$ of $K(a)|K$ and assume that $\#S_{F^\prime}(a) = 1$. Then 
	\[ \text{either } F\subseteq F^\prime \text{ or } F.F^\prime = K(a).  \]
\end{Proposition}

\begin{proof}
	Take the minimal polynomial $Q^\prime$ of $a$ over $F^\prime$. Observe that $(a,\g)$ is a minimal pair of definition for $w$ over $F^\prime$ by Lemma \ref{Lemma min pair of defn over K(z)}. Hence $j(Q^\prime) = [K(a):IC_{F^\prime}(w)]$ by Theorem \ref{Thm j = [K(a):ICF] min pair of defn}. Now $\o_{F^\prime}(a) \leq \o_K(a) = \g$. We first assume that $\o_{F^\prime}(a) = \g$. The condition $\#S_{F^\prime}(a) = 1$ then implies that $j(Q^\prime) = \deg Q^\prime$, that is, $IC_{F^\prime}(w) = F^\prime$. It follows that
	\[ F  = IC_{K}(w) \subseteq IC_{F^\prime}(w) = F^\prime.  \]
	We now consider the case when $\o_{F^\prime}(a) < \g$. Then $j(Q^\prime) = 1$ and hence $IC_{F^\prime}(w) = K(a)$. As a consequence, $K(a) \subseteq F^\prime(X)^h$. Employing [\ref{Kuh vln model}, Theorem 5.10], we conclude that $F^\prime(X)^h \subseteq K(a,X)^h = K(a).K(X)^h \subseteq F^\prime(X)^h$ and hence
	\[ F^\prime(X)^h = K(a,X)^h.  \]
	Observe that $F$ is relatively algebraically closed in $K(X)^h$ by definition. As a consequence, 
	\[ [F.F^\prime:F] = [K(X)^h.F^\prime : K(X)^h] = [F^\prime(X)^h:K(X)^h] = [K(a,X)^h : K(X)^h] = [K(a):F].   \]
	Since $F.F^\prime \subseteq K(a)$, we conclude that $F.F^\prime = K(a)$.
\end{proof}


\subsection{Estimate of $\# S_K(a)$} We now provide a bound for $\#S_K(a)$ whenever $a$ is pure over $K$. First we introduce a couple of notations which will be needed in the proof of the next theorem. Given a finite extension of fields $\mathcal{K}_2| \mathcal{K}_1$, we define 
\[ \ell(\mathcal{K}_2|\mathcal{K}_1) := \max \{ n \mid \text{there is a chain of subfields } \mathcal{K}_1 = L_0 \subsetneq L_1 \subsetneq \dotsc \subsetneq L_n = \mathcal{K}_2  \}.    \]
Also, given any valued field $(\mathcal{K}, \nu)$, we denote its henselization by $\mathcal{K}^{h(\nu)}$.

\begin{Theorem}\label{Thm S_K(a) bound a pure}
	Assume that $a$ is pure over $K$ and that $(K(a)|K,v)$ is not a tame extension. Then,
	\[ \#S_K(a) \leq \begin{cases*}
		v_p \deg_L(a) +1 \text{ whenever }\max v(a-K) \in S_K(a),\\
		v_p \deg_L(a) \text{ otherwise},
	\end{cases*}  \]
	where $v_p$ denotes the $p$-adic valuation and $L:= K(a)\sect K^r$.
\end{Theorem}

\begin{proof}
	The fact that $a$ is pure over $K$ implies that an Okutsu sequence of $a$ over $K$ is given by $A_0 = \{a\}, A_1$ where $A_1 \subset K$. Take any $z\in K$. Then $v(a-z) \leq v\left( \s(a-z) - (a-z) \right) = v(\s a - a) $ for all $\s\in \Gal(\overline{K}|K)$. It follows that
	\[ v(a-K) \leq  \min S_K(a) \leq \o_K(a) = \g.  \]
	If all the above inequalities are equalities, then $(K(a)|K,v)$ is tame by Lemma \ref{Lemma va = kras(a) tame}. Hence at least one of the inequalities is strict. We first assume that $v(a-K) < \min S_K(a) = \g$. Then $\#S_K(a) = 1$ and hence $(K(a)|K,v)$ is purely wild by Proposition \ref{Prop S_K(a) = 1 necessary condns}. It now remains to prove the assertion when $v(a-K)\leq \min S_K(a) < \g$.  
	
	\pars We write $S_K(a) = \{ \g, \g_1, \dotsc , \g_m \}$ where $\g > \g_1 > \dotsc > \g_m \geq v(a-K)$. The fact that $\g > \min S_K(a)$ implies that $m\geq 1$. Set $w_i := v_{a,\g_i}$ for all $i$. When $i = 0$, we identify $\g_i$ with $\g$ and $w_i$ with $w$. Since $\g_i > v(a-K)$, we observe that $(a,\g_i)$ is a minimal pair of definition for $w_i$ over $K$, for all $i\leq m-1$. Then 
	\[ j_{w_i}(Q) = [K(a):IC_K(w_i)] \text{ for all } i\leq m-1.  \]
	By definition, we have $j_w(Q) < j_{w_1}(Q)$. As a consequence, 
	\begin{equation}\label{Eqn 6}
		[IC_K(w):K] > [IC_K(w_1):K].
	\end{equation}
Take any $b\in K(a) \setminus \{0\}$ and the minimal polynomial $\tilde{Q}$ of $a$ over $K(b)$. Observe that $(a,\g)$ is a minimal pair of definition of $w_1$ over $K(b)$ by Lemma \ref{Lemma min pair of defn over K(z)}. It now follows from Theorem \ref{Thm j = [K(a):ICF] min pair of defn} that
	\begin{align*}
		b \in IC_K(w_1) &\Longleftrightarrow K(b,X)^{h(w_1)} = K(X)^{h(w_1)}\\
		&\Longleftrightarrow j_{w_1}(\tilde{Q}) = j_{w_1}(Q) \\
		&\Longleftrightarrow j_{w_1}(f) = 0,
	\end{align*}
	where $f:= Q/\tilde{Q} \in K(b)[X]$. Moreover, observe that $j_{w_1}(f) = 0$ implies that $j_w(f) = 0$ by definition. We have thus obtained that
	\[  b\in IC_K(w_1) \Longleftrightarrow j_{w_1}(f) = 0 \Longrightarrow j_w(f) = 0 \Longleftrightarrow b\in IC_K(w).   \]
	From (\ref{Eqn 6}), we conclude that
	\[IC_K(w_1) \subsetneq IC_K(w).\]
	Moreover, we observe that $L\subseteq IC_K(w_i)$ for all $i\leq m-1$ by Theorem \ref{Thm bound of ICF}. We thus have the chain
	\begin{equation}\label{Eqn 7}
		L\subseteq IC_K(w_{m-1})\subsetneq \dotsc \subsetneq IC_K(w_1) \subsetneq IC_K(w).
	\end{equation}
	It now follows from (\ref{Eqn 7}) that $m-1\leq \ell(IC_K(w)|L)$ and hence
	\begin{equation}\label{Eqn 8}
		\#S_K(a) = m+1 \leq \ell(IC_K(w)|L)+2. 
	\end{equation}
	Now $(K(a)|L,v)$ is a purely wild extension and hence $IC_K(w)|L$ is a $p$-power extension. It follows that
	\begin{equation}\label{Eqn 9}
		\ell(IC_K(w)|L) \leq v_p [IC_K(w):L].
	\end{equation}
	If $IC_K(w) = K(a)$, then $j(Q) = 1$ by Theorem \ref{Thm j = [K(a):ICF] min pair of defn}, which is not possible by our choice of $\g$. So $IC_K(w)$ is a proper subfield of $K(a)$ and hence $K(a)|IC_K(w)$ is a non-trivial $p$-power extension. Thus combining Equations (\ref{Eqn 8}) and (\ref{Eqn 9}), we obtain that
	\[ \#S_K(a) \leq \ell(IC_K(w)|L)+2 \leq v_p [IC_K(w):L] + 2 \leq v_p \deg_L(a)+1.   \]
	If $\g_m \neq \max v(a-K)$, then $\g_m > v(a-K)$ and hence $(a,\g_m)$ is a minimal pair of definition for $w_m$ over $K$. Hence in this case Equation (\ref{Eqn 7}) can be refined as 
	\[ L\subseteq IC_K(w_m)\subsetneq IC_K(w_{m-1})\subsetneq \dotsc \subsetneq IC_K(w_1) \subsetneq IC_K(w).  \]
	It then follows that
	\[ \# S_K(a) = m+1 \leq \ell(IC_K(w)|L) +1 \leq \ell(K(a)|L) \leq v_p \deg_L(a).  \]
\end{proof}

\begin{Corollary}\label{Coro S_K(a)=1 a pure K(a)|L minimal}
	Let notations and assumptions be as in Theorem \ref{Thm S_K(a) bound a pure}. Moreover, assume that $K(a)|L$ is a minimal extension. Then, 
	\[ \#S_K(a) = \begin{cases*}
		2 \text{ whenever } \max v(a-K) \in S_K(a),\\
		1 \text{ otherwise}.
	\end{cases*}  \]
\end{Corollary}

\begin{proof}
	It has been observed in the proof of Theorem \ref{Thm S_K(a) bound a pure}  that
	\[ \#S_K(a) \leq \ell(IC_K(w)|L) +2, \]
	with the inequality being strict whenever $v(a-K) < \min S_K(a)$. Moreover, $IC_K(w)$ is a proper subfield of $K(a)$. The minimality of $K(a)|L$ then implies that $IC_K(w) = L$. Hence the above expression can be rewritten as 
	\[  \#S_K(a) \leq 2,  \]
	 with the inequality being strict whenever $v(a-K) < \min S_K(a)$. If $\max v(a-K) = \min S_K(a)$ and $\#S_K(a) = 1$, then $(K(a)|K,v)$ is a tame extension by Lemma \ref{Lemma va = kras(a) tame}, contradicting our hypothesis. It follows that $\#S_K(a) = 2$ in this case.
\end{proof}


\section{Estimate of $\#S_K(a)$ in the general defectless case}

In the general defectless case, we have the following result:

\begin{Theorem}\label{Thm S_K(a) bound general}
	Assume that $(K(a)|K,v)$ is defectless. Further, assume that it is not a tame extension. Set
	\[ \epsilon := \min \{ i \mid (K(a_i)|K,v) \text{ is a tame extension} \}.  \]
	So $1\leq \epsilon \leq n$. Then,
	\[ K(a_{\epsilon}, \dotsc , a_n) = K(a_{\epsilon}) \subseteq L_i:= K(a_i)\sect K^r \text{ for all } i<\epsilon.  \]
	Moreover, 
	\[ \#S_K(a) \leq \sum_{i=0}^{\epsilon -1} v_p \deg_{L_i}(a_i) + n-\epsilon+1,  \]
	where $v_p$ is the $p$-adic valuation.
\end{Theorem}

\begin{proof}
	From Theorem \ref{Thm tame S_K(a)} we have that $K(a_{\epsilon}, \dotsc , a_n) = K(a_{\epsilon})$. Moreover, for any $i<\epsilon$, employing the triangle inequality we obtain that
	\[ v(a_i - a_\epsilon) = v(a_{\epsilon -1} - a_\epsilon) = \d_K(a_{\epsilon -1}) > \d_K(a_\epsilon) = \o_K(a_\epsilon),   \]
	where the final equality follows from Theorem \ref{Thm tame S_K(a)}. Then Lemma \ref{Lemma Krasner} yields that $a_\epsilon \in L_i$.  
	
	\pars We will employ backward induction to prove the final assertion. We first prove the base case for $i = \epsilon -1$. For any $\s\in G:= \Gal(\overline{K}|K)$, observe that
	\[  v(\s a_{\epsilon - 1} - a_{\epsilon - 1}) = v(\s a_{\epsilon - 1} - \s a_{\epsilon} + \s a_{\epsilon} - a_{\epsilon} + a_{\epsilon} - a_{\epsilon - 1}).  \]
	If $\s a_\epsilon \neq a_\epsilon$, then $v(\s a_\epsilon - a_\epsilon) \leq \o_K(a_\epsilon) = \d_K(a_\epsilon) < \d_K(a_{\epsilon - 1}) = v(a_{\epsilon - 1} - a_\epsilon)$. Hence $v(\s a_{\epsilon - 1} - a_{\epsilon - 1}) = v(\s a_\epsilon - a_\epsilon) \in S_K(a_\epsilon)$ in this case. Else $\s a_{\epsilon} = a_{\epsilon}$, and hence $v(\s a_{\epsilon - 1} - a_{\epsilon - 1}) \in S_{K(a_\epsilon)}(a_{\epsilon - 1})$. It follows that
	\[ \# S_K(a_{\epsilon - 1}) \leq  \#S_{K(a_\epsilon)} (a_{\epsilon - 1}) + \#S_K(a_\epsilon).   \]
	Now observe that $a_{\epsilon -1}$ is pure over $K(a_\epsilon)$. Moreover, $(K(a_{\epsilon -1})|K,v)$ is not tame by the definition of $\epsilon$. Consequently, $(K(a_\epsilon, a_{\epsilon -1})|K(a_\epsilon),v)$ is not a tame extension. Thus employing Theorem \ref{Thm S_K(a) bound a pure} and Theorem \ref{Thm tame S_K(a)} we obtain that
	\[ \# S_K (a_{\epsilon - 1}) \leq v_p \deg_{L_{\epsilon -1}} (a_{\epsilon - 1}) + n-\epsilon +1.   \]
	
	\parm We now assume that for some $0\leq i \leq \epsilon-1$ we have that
	\begin{equation}\label{Eqn 15}
		\#S_K(a_{i^\prime}) \leq \sum_{t=i^\prime}^{\epsilon -1} v_p \deg_{L_t}(a_t) + n-\epsilon+1, \text{ for all } i^\prime\geq i+1.
	\end{equation}
	Take some $\s\in G$ such that $\s a_i \neq a_i$. Write 
	\[ v(a_i - \s a_i) = v(a_i - a_{i+1} + a_{i+1} - \s a_{i+1} + \s a_{i+1} - \s a_i  ).  \]
	If $v(a_{i+1} - \s a_{i+1}) < \d_K(a_i)$ then $v(a_i - \s a_i) = v (a_{i+1} - \s a_{i+1}) \in S_K(a_{i+1})$ by the triangle inequality. In other words,
	\begin{equation}\label{Eqn 16}
		\text{Either }v(a_i - \s a_i) \geq \d_K(a_i) , \text{ or, } v(a_i - \s a_i) \in S_K(a_{i+1}). 
	\end{equation}
	We write 
	\[ S_K(a_i) = \{ \d_1, \dotsc , \d_{m^\prime} \},   \]
	where $\o_K(a_i) = \d_1 > \dotsc > \d_{m^\prime}$. Then $\d_1 \geq \d_K(a_i)$. Take the least positive integer $m \in \{ 1,\dotsc , m^\prime \}$ such that $\d_i < \d_K(a_i)$ for all $i > m$. We first assume that $m = 1$. It follows from (\ref{Eqn 16}) that $\d_i \in S_K(a_{i+1})$ for all $i\geq 2$. Hence, 
	\begin{equation}\label{Eqn 17}
		\# S_K(a_i) \leq \# S_K(a_{i+1}) + 1.
	\end{equation}
The fact that $i<\epsilon$ implies that $(K(a_i)|K,v)$ is not a tame extension. Consequently, $(K(a_i)|L_i,v)$ is a non-trivial purely wild extension and hence $v_p \deg_{L_i} (a_i) \geq 1$. Thus employing (\ref{Eqn 15}), we can modify (\ref{Eqn 17}) as
\[  \# S_K(a_i) \leq \sum_{t=i}^{\epsilon -1} v_p \deg_{L_t}(a_t) + n-\epsilon+1.   \]
We now assume that $m>1$. Set $w_s := v_{a_i, \d_s}$ for all $s\leq m$. Since $\d_s \geq \d_K(a_i)$, we obtain from Theorem \ref{Thm j = [K(a):ICF] key pol} that
\[  j_{w_s} (Q_i) = [K(a_i): IC_K(w_s)],  \]
where $Q_i$ is the minimal polynomial of $a_i$ over $K$. Employing the same arguments as in the proof of Theorem \ref{Thm S_K(a) bound a pure}, we now obtain a chain
\begin{equation}\label{Eqn 18}
	IC_K(w_1) \supsetneq IC_K(w_2) \supsetneq \dotsc \supsetneq IC_K(w_{m-1}) \supseteq L_i.
\end{equation}
From (\ref{Eqn 16}) we observe that 
\[ \#S_K(a_i) \leq m + \# S_K(a_{i+1}).   \]
Then (\ref{Eqn 18}) yields that
\begin{equation}\label{Eqn 19}
	\# S_K(a_i) \leq \ell (IC_K(w_1)|L_i) + 1 + \# S_K(a_{i+1}).
\end{equation}
If $IC_K(w_1) = K(a_i)$, then it follows from Theorem \ref{Thm j = [K(a):ICF] key pol} that $j_{w_1}(Q_i) = 1$ which is not possible since $\d_1 = \o_K(a_i)$. It follows that $IC_K(w_1)$ is a proper subfield of $K(a_i)$. As a consequence, we have that
\[ \ell (IC_K(w_1)|L_i) + 1 \leq \ell (K(a_i)|L_i) \leq v_p \deg_{L_i} (a_i).   \]
Then from (\ref{Eqn 15}) and (\ref{Eqn 19}) we conclude that
\[ \# S_K(a_i) \leq \sum_{t=i}^{\epsilon -1} v_p \deg_{L_t}(a_t) + n-\epsilon+1.   \]
The theorem now follows. 
\end{proof}


\section{Ramification ideals and the set $S_K(a)$}\label{Sect Ramification ideals}

\subsection{Ramification ideals}  Let $\mathcal{E}:= (\O|K,v)$ be a Galois extension of henselian valued fields. Set $\mathcal{G}:= \Gal(\O|K)$. For any subgroup $H$ of $\mathcal{G}$, we define
\[ \mathcal{I}_H:= \left( \dfrac{\s z}{z} - 1 \mid \s\in H, \, z\in \O^\times \right).  \]
Then $\mathcal{I}_H$ is an ideal of $\mathcal{O}_\O$.

\begin{Definition}
	An ideal $\mathcal{I}$ of $\mathcal{O}_\O$ is said to be a \textbf{ramification ideal} of $\mathcal{E}$ if $(0) \subsetneq \mathcal{I} \subseteq \mathcal{M}_\O$ and $\mathcal{I} = \mathcal{I}_H$ for some subgroup $H$ of $\mathcal{G}$. 
	\newline The set of ramification ideals of $\mathcal{E}$ will be denoted by $\Ram(\mathcal{E})$.
\end{Definition}

\begin{Remark}
	If $\mathcal{E}$ is a purely wild extension, then 
	\[ \Ram(\mathcal{E}) = \{ \mathcal{I}_H \mid H \neq \{\id  \} \}.    \]
\end{Remark}



%


\subsection{Valuation basis} 

\begin{Definition}
	Let $(\O|K,v)$ be a finite extension of valued fields. A $K$-basis $\{ x_1, \dotsc , x_n \}$ of $\O$ is said to form a \textbf{valuation basis} of $(\O|K,v)$ if for all $c_1, \dotsc , c_n$ in $K$, we have that $v(\sum_{i=1}^{n}  c_i x_i)  = \min\{  v(c_i x_i) \}$.
\end{Definition}

\begin{Remark}
	It is known (cf. [\ref{FVK-AR-PCK valn independence and defectless extns}, Proposition 3.4]) that a finite extension $(\O|K,v)$ of henselian valued fields admits a valuation basis if and only if it is defectless.
\end{Remark}

\begin{Proposition}\label{Prop depth one valn basis}
	Let $(b,0)$ be a distinguished pair over $(K,v)$ and set $n:= \deg_K(b)$. Then $\{1, b, \dotsc , b^{n-1}\}$ forms a valuation basis of $(K(b)|K,v)$.
\end{Proposition}

\begin{proof}
	Set $\tilde{w}:= v_{b,vb}$. From the given conditions we deduce that $\tilde{w}f = vf(b)$ for all $f(X)\in K[X]$ with $\deg f < n$. As a consequence, given $c_0, \dotsc , c_{n-1}$ in $K$ not all zero, we have that
	\begin{align*}
		 v \left( \sum_{i=0}^{n-1}c_i b^i \right)  = \tilde{w} \left( \sum_{i=0}^{n-1}c_i X^i \right) = v_{0, vb} \left( \sum_{i=0}^{n-1}c_i X^i \right) = \min\{ v(c_i b^i) \}.
	\end{align*}
Since $\{ 1, b, \dotsc , b^{n-1} \}$ forms a basis of $K(b)|K$, we have the assertion.
\end{proof}


\subsection{Some basic lemmas}

\begin{Lemma}\label{Lemma v(x^i-1) geq v(x-1)}
	Let $(K,v)$ be a valued field and $x\in \mathcal{O}_K$. Then $v(x^i-1) \geq v(x-1)$ for all $i\in\NN$.
\end{Lemma}

\begin{proof}
	The proof follows directly from the identity $x^i - 1 = (x-1)(x^{i-1}+ \dotsc + x +1)$ and the assumption $x\in\mathcal{O}_K$.
\end{proof}

\begin{Lemma}\label{Lemma v(sigma^i a - a) geq v(sigma a - a)}
	Let $(K,v)$ be a henselian valued field, $b\in\overline{K}\setminus\{0\}$ and $\s\in\Gal(\overline{K}|K)$. Then for all $i\in\NN$, we have that
	\[   v\left(\dfrac{\s^i b -b}{b}\right) \geq v\left(\dfrac{\s b - b}{b} \right).   \]
\end{Lemma}

\begin{proof}
	The assertion is immediate for $i=0,1$. Now assume that the assertion is true for some $i\in \NN$. Then $v(\s^i b -b) \geq v(\s b -b)$. Observe that $v(\s^{i+1}b - \s^i b) = v(\s b -b)$ and hence $v(\s^{i+1}b - b) \geq v(\s b -b)$ by the triangle inequality. The assertion now follows by induction. 
\end{proof}


\subsection{Depth one defectless extensions} 

\begin{Lemma}\label{Lemma I_H generating set}
	Let $\mathcal{E}:= (K(b^\prime)|K,v)$ be a defectless Galois extension of henselian valued fields. Assume that $b^\prime$ is pure over $K$. Then there exists a generator $b$ of $K(b^\prime)|K$ with $b$ pure over $K$ such that for any non-trivial subgroup $H$ of $\mathcal{G}:= \Gal(K(b^\prime)|K)$, we have that
	\[  \mathcal{I}_H = \left( \dfrac{\s b}{b} - 1 \mid \s\in H \right).   \]
\end{Lemma}

\begin{proof}
	Let $n:= \deg_K(b^\prime)$. If $(b^\prime, b_1^\prime)$ is a distinguished pair over $(K,v)$, we set $b:= b^\prime - b_1^\prime$. Then $(b,0)$ is a distinguished pair over $(K,v)$. Take $z\in K(b^\prime)^\times$. Then we have an expression $z = \sum_{i=0}^{n-1} c_i b^i$ where $c_i\in K$. For any $\s\in \mathcal{G}$, we have that
	\[ \dfrac{\s z}{z} - 1 = \sum_{i=0}^{n-1} \left( \dfrac{\s c_i b^i}{c_i b^i} - 1 \right) \dfrac{c_i b^i}{z} = \sum_{i=0}^{n-1} \left( \dfrac{\s b^i}{b^i} - 1 \right) \dfrac{c_i b^i}{z}.      \] 
	Since $v(c_i b^i) \geq vz$ by Proposition \ref{Prop depth one valn basis}, we have that $\frac{\s z}{z} - 1 \in \left( \frac{\s b^i}{b^i} - 1 \mid \s\in H, \, 0\leq i \leq n-1 \right)$. Now Lemma \ref{Lemma v(x^i-1) geq v(x-1)} yields that $v \left( \frac{\s b^i}{b^i} - 1 \right) \geq v \left( \frac{\s b}{b} - 1 \right)$. We thus conclude that
	\[  \dfrac{\s z}{z} - 1 \in \left( \dfrac{\s b}{b} - 1  \mid \s\in H  \right).  \]
	The assertion now follows. 
\end{proof}

We can now prove the following theorem:

\begin{Theorem}\label{Thm Ram(E) = S_K(a)}
	Assume that $\mathcal{E}:= (K(a)|K,v)$ is defectless purely wild and Galois. Moreover, assume that $a$ is pure over $K$. Then we have that
	\[  \# \Ram(\mathcal{E}) = \# S_K(a).  \]
\end{Theorem}

\begin{proof}
	Since $S_K(a) = S_K(a-a_1)$ for any $a_1\in K$, we can assume without any loss of generality that $(a,0)$ is a distinguished par over $(K,v)$. Let $\# S_K(a) = m$. Take $\s_1, \dotsc , \s_m \in \Gal(L|K)$ such that $S_K(a) = \{ v(\s_1 a -a), \dotsc , v(\s_m a -a) \}$. Denote by $H_i$ the subgroup generated by $\s_i$. It follows from Lemma \ref{Lemma I_H generating set} that 
	\[ \mathcal{I}_{H_i} = \left( \dfrac{\s_i ^j a}{a} - 1 \mid j\in\NN\right).  \]
	From Lemma \ref{Lemma v(sigma^i a - a) geq v(sigma a - a)}, we have that $v(\s_i ^j a - a) \geq v(\s_i a -a)$. As a consequence, we obtain that
	\[  \mathcal{I}_{H_i} = \left( \dfrac{\s_i a}{a} - 1 \right).  \]
	Moreover for $i\neq j$ we have that $v(\s_i a -a) \neq v(\s_j a -a)$. Thus $\mathcal{I}_{H_i} \neq \mathcal{I}_{H_j}$. It follows that $\#\Ram(\mathcal{E}) \geq m$. 
	
	\pars Conversely, take any non-trivial subgroup $H$ of $\Gal(K(a)|K)$. By Lemma \ref{Lemma I_H generating set}, we have that 
	\[ \mathcal{I}_H = \left( \dfrac{\s a }{a} - 1 \mid \s\in H \right).  \]
	Since $H$ is finite, we can choose $\s_H \in H$ such that $v(\s_H a - a) \leq v(\s a -a )$ for all $\s\in H$. It follows that
	\[  \mathcal{I}_H = \left( \dfrac{\s_H a}{a} - 1 \right).  \]
	Observe that $v(\s_H a -a) \in S_K(a)$ and hence $v(\s_H a - a) = v(\s_i a -a)$ for some $i\in\{1, \dotsc , m\}$. As a consequence, $\mathcal{I}_H = \mathcal{I}_{H_i}$. The theorem now follows.
\end{proof}

We have the following result in the general case:

\begin{Corollary}\label{Coro Ram(E) = S_L(a)}
	Assume that $\mathcal{E}:= (K(a)|K,v)$ is a defectless Galois extension which is not tame. Further assume that $a$ is pure over $K$. Then,
	\[  \# \Ram(\mathcal{E}) = \# S_L(a), \text{ where } L:= K(a)\sect K^r.  \]
\end{Corollary}

\begin{proof}
    Since $(K(a)|L,v)$ is a purely wild extension, in light of Theorem \ref{Thm Ram(E) = S_K(a)} it is sufficient to show that $\Ram(\mathcal{E}) = \Ram(K(a)|L,v)$. Clearly, $\Ram(K(a)|L,v) \subseteq \Ram(\mathcal{E})$. For the converse, take any subgroup $H$ of $\Gal(K(a)|K)$ such that $H\nsubseteq \Gal(K(a)|L)$. Then there exists some $\s\in H$ such that $\s x \neq x$ for some $x\in L$. Since $L = K(a)\sect K^r$, this implies that $\s\notin G^r$. As a consequence, there exists some $z\in K(a)^\times$ such that $v(\s z - z) = vz$. In other words, $\frac{\s z}{z} - 1$ is a unit in $\mathcal{O}_{K(a)}$. It follows that $\mathcal{I}_H = \mathcal{O}_{K(a)}$ and hence $\mathcal{I}_H\notin \Ram(\mathcal{E})$.
\end{proof}

\pars The conclusions of Theorem \ref{Thm Ram(E) = S_K(a)} and Corollary \ref{Coro Ram(E) = S_L(a)} are not true for arbitrary defectless Galois extensions, as evidenced by Example \ref{Eg S_K(a) neq S_K(b)}. A possible avenue for generalizing them is by introducing the notion of \textit{depth of an extension}, defined below. In the setup of Example \ref{Eg S_K(a) neq S_K(b)}, setting $\g^\prime:= \b-\a$, we observe that $K(\b) = K(\g^\prime)$ and $\ell(\g^\prime) = 1 < 2 = \ell(\b)$. This motivates the following definition:

\begin{Definition}
	Let $\mathcal{E}:= (F|K,v)$ be a defectless extension of henselian valued fields. Then we define
	\[ \depth(\mathcal{E}) := \min\{ \ell(a) \mid a \text{ is a generator of } F|K  \}.   \]
\end{Definition}

The following theorem is an immediate consequence of Theorem \ref{Thm Ram(E) = S_K(a)} and Corollary \ref{Coro Ram(E) = S_L(a)}:

\begin{Theorem}\label{Thm depth one defectless Ram(E) = S_K(a)}
	Let $\mathcal{E}:= (\O|K,v)$ be a defectless Galois extension of henselian valued fields. Assume that $\depth(\mathcal{E}) = 1$. Further assume that $\mathcal{E}$ is not tame. Then for any generator $a$ of $\O|K$ such that $a$ is pure over $K$, we have that
	\[ \#\Ram(\mathcal{E}) = \# S_L(a), \text{ where } L:= \O\sect K^r.  \]
	In particular, if $\mathcal{E}$ is purely wild, we have that
	\[ \#\Ram(\mathcal{E}) = \# S_K(a).  \]
\end{Theorem}

In light of the previous theorem, we propose the following problem:

\begin{Question}
	Let $\mathcal{E}:= (\O|K,v)$ be a defectless Galois purely wild extension of henselian valued fields and assume that $\depth(\mathcal{E}) = n$. Does there always exist a generator $a$ of $\O|K$ with $\ell(a) = n$ such that $\#\Ram(\mathcal{E}) = \# S_K(a)$? If yes, is it true for every generator of minimal depth?
\end{Question}


\subsection{Depth one defect extensions} We now consider the case when $\mathcal{E}:= (\O|K,v)$ is a Galois defect extension of depth one. The following theorem generalizes [\ref{FVK AR Topics in higher ramification theory}, Corollary 3.11].

\begin{Theorem}\label{Thm I_H defect depth one}
	Assume that $\mathcal{E}:= (K(a)|K,v)$ is a Galois defect extension of henselian valued fields. Further assume that $a$ is pure over $K$. Then for every subgroup $H$ of $\mathcal{G}:= \Gal(K(a)|K)$, we have that
	\[ \mathcal{I}_H = \left( \dfrac{\s a - a}{a - z_\nu} \mid \s\in H, \, \nu<\l\right),  \]
	where $\{z_\nu\}_{\nu<\l}$ is a pcs in $K$ without a limit in $K$ and which has $a$ as a limit. 
\end{Theorem}

\begin{proof}
	We consider a pcs $\{z_{\nu}\}_{\nu<\l}$ in $K$ with $a$ as a limit, as observed in Remark \ref{Rmk characterization pure elts}. Take any $b\in K(a)^\times$. Then $b = f(a)$ for some $f(X)\in K[X]$ with $\deg f < \deg_K(a)$. Consider the Hasse-Schmidt derivatives $\partial_i f(X) \in K[X]$ satisfying
	\[  f(X) - f(z) = \sum_{i=1}^{d} \partial_i f(z) (X-z)^i \text{ for all } z\in\overline{K},  \]
	where $d = \deg f$. Take some $\s\in\mathcal{G}$. Then, 
	\begin{align*}
		\s b - b &= \sum_{i=1}^{d} \partial_i f(a) (\s a -a)^i,\\
		f(a) - f(z_\nu) &= \sum_{i=1}^{d} \partial_i f(z_\nu) (a-z_\nu)^i, \text{ for all } \nu<\l.
	\end{align*} 
Observe that $\deg\partial_i f(X) < \deg_K(a)$ for all $i$. In light of Remark \ref{Rmk characterization pure elts} and [\ref{Kaplansky}, Lemma 4], we can choose some $h\in \{1, \dotsc , d\}$ and some ordinal $\nu_0<\l$ large enough such that the following conditions are satisfied:
\begin{condenum}
	\item $vf(a) = vf(z_\nu)$ for all $\nu_0 \leq \nu < \l$,\vspace{2 mm} \label{Cond 1}
	\item $v \partial_i f(z_\nu) = v\partial_i f(a)$ for all $\nu_0\leq \nu<\l$ and $i=1, \dotsc , d$, \vspace{2 mm}\label{Cond 2} 
	\item $ v\left( \partial_h f(z_\nu) (a-z_\nu)^h \right) < v \left(\partial_i f(z_\nu) (a-z_\nu)^i\right)$ for all $i\neq h$ and $\nu_0\leq \nu<\l$. \label{Cond 3}
\end{condenum}
Take some $j\in\{1, \dotsc , d\}$ such that 
\[ v\left( \partial_j f(a) (\s a-a)^j \right) \leq v\left( \partial_i f(a) (\s a-a)^i \right) \text{ for all } i\neq j.  \]
Then, for all $\nu_0\leq\nu<\l$, we obtain that
\begin{align*}
	 v(\s b - b) - vb & \geq v\left( \partial_j f(a) (\s a-a)^j \right) - vb \\
	 &= v\left( \partial_j f(z_\nu)(a-z_\nu)^j\right) -vf(a) + jv\left( \dfrac{\s a-a}{a-z_\nu} \right) \text{ by Condition \textbf{C2}}\\
	 &\geq v\left( \partial_j f(z_\nu)(a-z_\nu)^j\right) -v\left( f(a) - f(z_\nu)\right) + jv\left( \dfrac{\s a-a}{a-z_\nu} \right) \text{ by Condition \textbf{C1}}\\
	 &\geq jv\left( \dfrac{\s a-a}{a-z_\nu} \right) \text{ by Condition \textbf{C3}}.
\end{align*}
Since $z_\nu\in K$, we have that $v(\s a - z_\nu) = v(a-z_\nu)$, and hence $v(\s a-a)\geq v(a-z_\nu)$ for all $\nu<\l$. The sequence $\{v(a-z_\nu)\}_{\nu<\l}$ being strictly increasing, we conclude that
\[ v(\s a-a) > v(a-z_\nu) \text{ for all } \nu<\l.  \]
As a consequence, we obtain that
\[ v(\s b - b) - vb \geq jv\left( \dfrac{\s a-a}{a-z_\nu} \right) \geq v \left( \dfrac{\s a-a}{a-z_\nu}\right). \]
Observe that $\dfrac{\s a -a}{a-z_\nu} = \dfrac{\s(a-z_\nu) - (a-z_\nu)}{a-z_\nu} \in \mathcal{I}_H$ whenever $\s\in H$. The theorem now follows.
\end{proof}

As an immediate consequence, we obtain the following: 

\begin{Theorem}\label{Thm Ram(E) leq S_K(a) defect}
	Let $\mathcal{E}:= (\O|K,v)$ be a Galois defect extension of henselian valued fields. Assume that $\depth(\mathcal{E}) = 1$. Then for any generator $a$ of $\O|K$ such that $a$ is pure over $K$, we have that
	\[ \#\Ram(\mathcal{E}) \leq \# S_K(a).  \]
\end{Theorem}

\begin{proof}
	Take a generator $a$ of $\O|K$ which is pure over $K$. We also choose a pcs $\{z_\nu\}_{\nu<\l}$ in $K$ with $a$ as a limit as per Remark \ref{Rmk characterization pure elts}. Write $S_K(a) = \{ v(\s_1 a -a), \dotsc , v(\s_m a - a) \}$ where $\s_1, \dotsc , \s_m \in \mathcal{G}:= \Gal(\O|K)$. Denote by $H_i$ the subgroup of $\mathcal{G}$ generated by $\s_i$. It then follows from Lemma \ref{Lemma v(sigma^i a - a) geq v(sigma a - a)} and Theorem \ref{Thm I_H defect depth one} that
	\[ I_{H_i} =  \left( \dfrac{\s_i a - a}{a - z_\nu} \mid  \nu<\l\right). \]
	Take a subgroup $H$ of $\mathcal{G}$. Since $H$ is finite, we can choose $\s_H \in H$ such that $v(\s_H a - a) \leq v(\s a - a)$ for all $\s\in H$. Moreover, choose $i\in\{1, \dotsc , m\}$ such that $v(\s_H a - a) = v(\s_i  a - a)$. It then follows from Theorem \ref{Thm I_H defect depth one} and our preceding discussions that
	\[ \mathcal{I}_H = \left( \dfrac{\s_H a - a}{a - z_\nu} \mid  \nu<\l\right) = \mathcal{I}_{H_i}.   \]
	The assertion thus follows.
\end{proof} 

Unlike the defectless case, it is not evident whether we have equality in Theorem \ref{Thm Ram(E) leq S_K(a) defect}. An easy case when this holds true is when $[\O:K] = p$, since then $\# \Ram(\mathcal{E}) = 1 = \# S_K(a)$ by Theorem \ref{Thm S_K(a)=1 minimal extn}. We now show that equality also holds whenever $\rk (K,v) = 1$.

\begin{Theorem}\label{Thm Ram(E) = S_K(a) defect rank one}
	Let notations and assumptions be as in Theorem \ref{Thm Ram(E) leq S_K(a) defect}. Further assume that $\rk (K,v) = 1$. Then we have
	\[ \#\Ram(\mathcal{E}) = \#S_K(a), \]
	for every generator $a$ of $\O|K$ which is pure over $K$.
\end{Theorem}

\begin{proof}
	We use the same notations as in the proof of Theorem \ref{Thm Ram(E) leq S_K(a) defect}. It is enough to show that $\mathcal{I}_{H_i} \neq \mathcal{I}_{H_j}$ for all $i\neq j$. Suppose the contrary, that is, $\mathcal{I}_{H_i} = \mathcal{I}_{H_j}$ for some $i\neq j$. Employing Lemma \ref{Lemma v(sigma^i a - a) geq v(sigma a - a)} and Theorem \ref{Thm I_H defect depth one}, we obtain that
	\begin{equation}\label{eqn_1}
		\mathcal{I}_{H_i} =  \left( \dfrac{\s_i a - a}{a - z_\nu} \mid  \nu<\l\right) = \left( \dfrac{\s_j a - a}{a - z_\nu} \mid  \nu<\l\right) = \mathcal{I}_{H_j}.
	\end{equation}
	Assume that $v(\s_i a - a) < v(\s_j a - a)$. Then $v\left( \dfrac{\s_i a - a}{a - z_\nu} \right) < v\left( \dfrac{\s_j a - a}{a - z_\nu}\right)$ for all $\nu<\l$. In light of (\ref{eqn_1}), we then conclude that for every $\nu<\l$, there exists some $\mu<\l$ such that
	\[ v\left( \dfrac{\s_j a - a}{a - z_\mu}\right) < v\left( \dfrac{\s_i a - a}{a - z_\nu}\right).  \]
	In other words, for every $\nu<\l$, there exists some $\mu<\l$ such that
	\begin{equation}\label{eqn_2}
		v(a-z_\mu) - v(a-z_\nu) > v(\s_j a - a) - v(\s_i a - a).
	\end{equation}
Set $\d:= v(\s_j a - a) - v(\s_i a - a) > 0$. Since $v(a-K)$ is an initial segment of $vK$, repeated applications of (\ref{eqn_2}) and the Archimedean property of $vK$ yield that $v(a-K) = vK$. In other words, we obtain that $a$ lies in the completion $\widehat{K}$ of $(K,v)$. However, the facts that $\rk(K,v) = 1$ and $(K,v)$ is henselian imply that $K$ is separable-algebraically closed in $\widehat{K}$. As a consequence, we must have $a\in K$, which contradicts our assumption that $a\notin K$. We conclude that $\mathcal{I}_{H_i} \neq \mathcal{I}_{H_j}$ for all $i\neq j$, and hence $\#\Ram(\mathcal{E}) = \#S_K(a)$ by Theorem \ref{Thm Ram(E) leq S_K(a) defect}.
\end{proof}

A variant of the above argument gives another criterion for equality to hold in the statement of Theorem \ref{Thm Ram(E) leq S_K(a) defect} for a \textit{particular} pure generator. 

\begin{Definition}
	An initial segment $\Sigma$ of an ordered set $\O$ is said to be \textbf{principal} if $\O\setminus \Sigma$ has a minimal element, that is, there exists some $\g\in\O$ such that $\Sigma = \{ \g^\prime\in\O\mid\g^\prime<\g \}$.
\end{Definition}

\begin{Proposition}
	Let $\mathcal{E}=(K(a)|K,v)$ be a Galois defect extension of henselian valued fields. Assume that $a$ is pure over $K$. Further assume that $v(a-K)$ is principal. Then,
	\[ \#\Ram(\mathcal{E}) = \# S_K(a).  \]
\end{Proposition}

\begin{proof}
	We use the same notations as in the proof of Theorem \ref{Thm Ram(E) leq S_K(a) defect}. As observed in the proof of Theorem \ref{Thm Ram(E) = S_K(a) defect rank one}, it is enough to show that $\mathcal{I}_{H_i} \neq \mathcal{I}_{H_j}$ for all $i\neq j$. If equality holds for some $i\neq j$, then assuming $\d:= v(\s_j a - a) - v(\s_i a - a) > 0$, we obtain in (\ref{eqn_2}) that for each $\nu<\l$, there exists some $\mu<\l$ such that
	\[   v(a-z_\mu) - v(a-z_\nu) > \d.  \]
	Since $v(a-K)$ is principal, we can take some $\g\in vK$ such that $v(a-K) = \{ \g^\prime\in vK \mid \g^\prime<\g \}$. Then $\g-\d\in v(a-K)$. The cofinality of the sequence $\{v(a-z_\nu)\}_{\nu<\l}$ in $v(a-K)$ implies that there exists some $\nu<\l$ such that $\g-\d \leq v(a-z_\nu) < \g$. But this contradicts (\ref{eqn_2}). 
\end{proof}


\section{Some examples}

The following example illustrates that the cardinality of the set $\# S_K(a)$ depends on the element $a$ and not on the field extension $(K(a)|K,v)$. We will exhibit this by constructing an extension with two different generators $a$ and $b$ such that $\#S_K(a) \neq \#S_K(b)$. Moreover, we will construct the extension to be defectless and Galois. This will demonstrate that the conclusions of Theorem \ref{Thm Ram(E) = S_K(a)} and Corollary \ref{Coro Ram(E) = S_L(a)} are not true for arbitrary defectless Galois extensions.

\begin{Example}\label{Eg S_K(a) neq S_K(b)}
	Let $k$ be field with $\ch k = p>2$. Assume that $k$ is Artin-Schreier closed. Set $(K,v)$ to be the henselization of $k(t)$ equipped with the $t$-adic valuation. Take $\a,\b \in \overline{K}$ such that $\a^p - \a = 1/t$ and $\b^p - \b = \a + 1/t$. Then
	\[  v\a = -\dfrac{1}{p}, \, vK(\a) = \dfrac{1}{p}\ZZ \text{ and } \deg_K(\a) = p.  \]
	Now $\a\in K(\b)$ and hence $\deg_K(\b) \leq p^2$. Observe that $(\b-\a)^p - (\b-\a) = \a$ and hence $v(\b-\a) = -1/p^2$. It follows that $\b-\a \notin K(\a)$. As a consequence, 
	\[ K(\a) \subsetneq K(\b) \text{ and } \deg_K(\b) = p^2.  \]
	Take $\tau\in \Gal(\overline{K}|K)$ such that $\tau\a = \a+1$. Set $b:= \a+ 1/t$. Then $\tau b - b = 1$. The base field $k$ being Artin-Schreier closed, the polynomial $X^p - X - 1$ splits completely over $K$. It now follows from [\ref{Nart Novacoski depth of A-S defect towers}, Lemma 2.9] that 
	\[ K(\b)|K \text{ is a Galois extension}. \] 
	Set $\mathcal{G}:= \Gal(K(\b)|K)$. For any $\s\in\mathcal{G}$ we have that $\s\b^p - \s\b = \s\a + 1/t$ and hence 
	\[ (\s\b -\b)^p - (\s\b-\b) = \s\a-\a.  \]
	The fact that $\a$ is an Artin-Schreier generator over $K$ implies that $\s\a-\a\in\FF_p$. Consequently, $\s\b-\b \in k$ and hence $v(\s\b-\b) = 0$ for all $\s\in\mathcal{G}$. We have thus shown that
	\[ \#S_K(\b) = 1.  \]
	Set $\g:= \b-\a^2$. Then $\g\in K(\b)$. Observe that
	\[ \g^p - \g = \dfrac{1}{t} - \dfrac{1}{t^2} + \left( 1-\dfrac{2}{t} \right) \a.  \]
	Thus $\a\in K(\g)$. Moreover, the fact that $\b\notin K(\a)$ implies that $\g\notin K(\a)$. Thus
	\[ K(\b) = K(\g).  \]
	Take any $\s\in\mathcal{G}$. Then $\s\g = \s\b - \s\a^2$ and hence
	\[ \s\g - \g = (\s\b-\b) - (\s\a^2 - \a^2).   \]
	Observe that $\s\a-\a = i \in \FF_p$. If $i = 0$ then $v(\s\g-\g) = v(\s\b-\b)=0$. Else, $\s\g-\g = (\s\b-\b) - (2i\a+i^2)$, and hence $v(\s\g-\g) = -1/p$ by the triangle inequality. Thus
	\[ \#S_K(\g) = 2. \]
\end{Example}

\pars This next example illustrates that the inequality in the statement of Theorem \ref{Thm S_K(a) bound a pure} may be strict. 

\begin{Example}\label{Eg strict inequality in Thm S_K(a) bound pure}
	Let $k$ be a field with $\ch k = p > 0$ such that $k$ is not perfect. Denote by $(K,v)$ the henselization of $k(t)$ equipped with the $t$-adic valuation. Then $vK = \ZZ$ and $Kv = k$. We denote by $v$ the extension of $v$ to $\overline{K}$ as well. Take $c\in k\setminus k^p$. Let $\th$ be a root of the polynomial 
	\[  X^p - t^{p-1} X - c \in K[X].   \]
	Since $vc = 0$, this implies that $v(\th^p - t^{p-1}\th) =0$. If $v\th > 0$ then $v(\th^p - t^{p-1}\th)>0$. On the other hand, if $v\th <0$ then $v\th^p < v(t^{p-1}\th)$ and hence $v(\th^p - t^{p-1}\th) = v\th^p <0$ by the triangle inequality. It follows that $v\th = 0$. Then $\th v$ is a root of the polynomial $X^p - c \in k[X]$. The fact that $c\notin k^p$ implies that $\th v$ is purely inseparable over $k$ with $\deg_k (\th v) = p$. As a consequence, we conclude from the Fundamental Inequality that
	\[ \deg_K(\th) = \deg_k(\th v) = p, \, vK(\th) = \ZZ, \, K(\th)v = k(\th v).  \]
	Take a root $a$ of the polynomial
	\[  f(X):= X^p - X - \dfrac{1+\th}{t} \in K(\th)[X].  \]
	Then $va = -1/p$ by the triangle inequality. Employing the Fundamental Inequality again, we conclude that
	\[  \deg_{K(\th)}(a) = (vK(\th,a):vK(\th)) = p, \, vK(\th,a) = \dfrac{1}{p}\ZZ, \, K(\th,a)v = K(\th)v = k(\th v).  \]
	As a consequence, $[K(\th,a):K] = p^2$. Now observe that $\th\in K(a)$ by definition. So either $\deg_K(a) = p^2$ or $K(\th) = K(a)$. However the latter is not possible since $-1/p = va \in vK(a)\setminus vK(\th)$. It follows that $K(a) = K(\th,a)$ and hence
	\begin{equation}\label{Eqn 10}
		\deg_K(a) = p^2, \, (vK(a):vK) = p, \, K(a)v = k(\th v) \text{ is purely inseparable over }Kv.
	\end{equation}
	Thus $(K(a)|K,v)$ is a defectless extension and hence $a$ admits a complete distinguished chain over $K$. Take some $z\in\overline{K}$ such that $v(a-z)> va$. Then $va = vz  = -1/p$ and $\dfrac{a}{z}v = 1v$. Observe that
	\[ \dfrac{a^p}{z^p} - \dfrac{1}{z^{p-1}}\dfrac{a}{z} = \dfrac{1+\th}{t z^p}.  \]
	Taking residues and observing that $ 0< v \dfrac{1}{z^{p-1}}$, we obtain that
	\[ \dfrac{a^p}{z^p}v = \dfrac{1+\th}{t z^p}v.  \]
	Since $\dfrac{a}{z}v = 1v$ and $v(tz^p) = 0 \leq v(1+\th)$, we have that
	\[ (tz^p)v = (1+\th)v = 1v+\th v. \]
	It follows that $\th v \in K(z)v$ and hence $[K(z)v : Kv] \geq p$. Moreover, the fact that $vz = -1/p$ implies that $(vK(z):vK) \geq p$. From the Fundamental Inequality, we conclude that $\deg_K(z) \geq p^2 = \deg_K(a)$. In other words, 
	\[ \deg_K(z) < \deg_K(a) \Longrightarrow v(a-z) \leq va = \max v(a-K),  \]
	where the last assertion follows from the triangle inequality. We have thus shown that $(a,0)$ is a distinguished pair and hence
	\begin{equation}\label{Eqn 11}
		\ell(a) = 1.
	\end{equation}
	We now take some $\s\in G:= \Gal(\overline{K}|K)$ such that $\s a \neq a$. If $\s\th = \th$, then $\s \in \Gal(\overline{K}|K(\th))$ and hence $\s a$ is a $K(\th)$-conjugate of $a$. Since the minimal polynomial $f(X)$ of $a$ over $K(\th)$ is an Artin-Schreier polynomial, we conclude that $\s a - a \in \FF_p^{\times}$ and hence
	\begin{equation}\label{Eqn 12}
		v(\s a - a) = 0.
	\end{equation}
	If $\s \th \neq \th$, then we consider the following equations:
	\begin{align*}
		a^p - a &= \dfrac{1+\th}{t},\\
		\s a^p - \s a &= \dfrac{1 + \s \th}{t}.
	\end{align*} 
	As a consequence, 
	\begin{equation}\label{Eqn 13}
		(\s a -a)^p - (\s a - a) = \dfrac{\s\th - \th}{t}.
	\end{equation}
	Again, consider the following equations:
	\begin{align*}
		\th^p - t^{p-1}\th &= c,\\
		\s\th^p - t^{p-1}\s\th &=c.
	\end{align*}
	It follows that $(\s\th - \th)^p = t^{p-1}(\s\th - \th)$ and hence $v(\s\th - \th) = 1$. Thus $v \dfrac{\s\th - \th}{t} = 0$. Employing the triangle inequality in (\ref{Eqn 13}), we obtain that
	\begin{equation}\label{Eqn 14}
		v(\s a - a) = 0.
	\end{equation}
	From Equations (\ref{Eqn 12}) and (\ref{Eqn 14}) we conclude that $\# S_K(a) = 1$. Moreover $(K(a)|K,v)$ is purely wild by (\ref{Eqn 10}) and $a$ is pure over $K$ by (\ref{Eqn 11}). Thus,
	\[ \# S_K(a) = 1 < 2 = v_p \deg_K(a),  \]
	which shows that the inequality in Theorem \ref{Thm S_K(a) bound a pure} is not an equality. 
\end{Example}

\pars Modifying the above example slightly, we construct an example that illustrates that the bounds obtained in Theorem \ref{Thm S_K(a) bound a pure} and Corollary \ref{Coro S_K(a)=1 a pure K(a)|L minimal} are sharp. It also illustrates that the answer to Problem \ref{Qn pure implies S_K(a)=1} is not positive in general, even in the defectless case.

\begin{Example}\label{Eg bound sharp a pure}
	Let $k, \, (K,v)$ and $\th$ be as in Example \ref{Eg strict inequality in Thm S_K(a) bound pure}. Assume further that $k$ is not Artin-Schreier closed. Take $d\in k$ such that $X^p - X - d$ is irreducible over $k$. Take $\a\in\overline{K}$ such that $\a^p - \a = d$. Then $v\a = 0$ and $\a v$ is a root of the Artin-Schreier polynomial $X^p-X-d$. It follows that $K(\a)v= k(\a v)$ is a separable degree $p$ extension over $Kv$. Thus, 
	\[  (K(\a)|K,v) \text{ is a tame extension of degree }p. \]
	It was further observed in Example \ref{Eg strict inequality in Thm S_K(a) bound pure} that $K(\th)v = k(\th v)$ is a purely inseparable degree $p$ extension over $Kv$ and hence
	\[ (K(\th)|K,v) \text{ is a purely wild extension of degree $p$}.  \]
	Hence $[K(\th,\a):K] = p^2$ by Lemma \ref{Lemma tame purely wild valn disjoint}. We set
	\[  a:= \th+\a. \]
	Then $a^p - a = c+d+(t^{p-1}-1)\th$. As a consequence, $\th\in K(a)$. It follows that $\a\in K(a)$ and hence $K(a) = K(\th,\a)$. Thus $\deg_K(a) = p^2$. Moreover, it was also observed in Example \ref{Eg strict inequality in Thm S_K(a) bound pure} that $v(\s\th-\th) = 1$ for any $\s\in G:= \Gal(\overline{K}|K)$ not fixing $\th$. On the other hand, if $\s\a\neq \a$, then the fact that $\a$ satisfies an Artin-Schreier polynomial yields that $v(\s\a-\a) = 0$. It follows that $S_K(a) = \{0,1\}$ and hence 
	\[ \# S_K(a) = 2.  \]
	Observe that $L:= K(a)\sect K^r = K(\a)$ and hence 
	\[  \#S_K(a) = v_p \deg_L(a)+1. \]
	We are thus done if we can show that $a$ is pure over $K$. Observe that $v\th = 0 = v\a$ and hence $va\geq 0$. Take some $z\in\overline{K}$ such that $v(a-z) > va$. Then $av = zv$. Consider the relation $a^p - a = c+d+(t^{p-1}-1)\th$. Taking residues and employing the condition $av = zv$, we obtain that
	\[ zv^p - zv = (c+d)v -\th v.  \]
	As a consequence, $\th v \in K(z)v$. We also have the relation $a^p - t^{p-1}a = (c+d) + (1-t^{p-1})\a$. Taking residues we obtain that
	\[  zv^p = (c+d)v + \a v, \]
	and hence $\a v\in K(z)v$. It follows that
	\[ \deg_K(z) \geq \deg_{Kv}(zv) \geq [Kv(\a v, \th v):Kv] = p^2, \]
	where the last equality follows from the linear disjointness of the separable extension $Kv(\a v)|Kv$ and the purely inseparable extension $Kv(\th v)|Kv$. We have thus obtained that
	\[ v(a-z)>vz \Longrightarrow \deg_K(z) \geq p^2 = \deg_K(a).  \]
	Thus $(a,0)$ is a distinguished pair and hence $a$ is pure over $K$.  
\end{Example}

\begin{Remark}
	We can construct imperfect fields which are not Artin-Schreier closed in the following way. Take a field $F$ which is not Artin-Schreier closed and a variable $X$. Then $F(X)$ has the desired properties. 
\end{Remark}

\appendix
\renewcommand\thesection{A}
\section{}\label{sec:appendix_section}
\renewcommand{\theappendixlemma}{\Alph{section}.\arabic{appendixlemma}}

We want to present some relations of the $j$-invariant with other objects in the literature. For that purpose, we will introduce some notation and definitions. Let $v$ be a valuation on $\overline K$ and set $\Gamma=v\overline K$.

For all $s\in \NN$, the $s$-th Hasse-Schmidt derivative $\partial_s$ on $K[x]$ is defined by:
$$
f(x+y)=\sum\nolimits_{0\le s}(\partial_s f)y^s \quad\mbox{for all } \,f\in K[x],
$$
where $y$ is another indeterminate.

Take a valuation $\mu$ on $K[x]$. If $f\in K[x]$ is non-constant, we define
$$
\epm(f)=\max\left\{\dfrac{\mu(f)-\mu(\ps{f})}s\ \Big|\ s\in\NN\right\}.
$$
For $f\in K[x]$ we denote
\[
I_\mu(f)=\left\{s\in\{1,\ldots,\deg(f)\}\mid \epsilon_\mu(f)=\dfrac{\mu(f)-\mu(\ps{f})}s\right\}.
\]

Let us recall [\ref{Novacoski key poly and min pairs}, Proposition 3.1], which clarifies the meaning of $\epm(f)$.

\begin{Theorem}\label{nov}
	Let $\overline\mu$ be an arbitrary extension  to $\kbx$ of $\mu$. Then, for every non-constant $f\in \kx$ we have 
	$$
	\epm(f)=\max\{\overline\mu(x-a)\mid a\in \op{Z}(f)\},
	$$
	where $\op{Z}(f)$ is the set of roots of $f$ in $\kb$. 
\end{Theorem}

We say that $a\in \op{Z}(f)$  is an \textbf{optimizing root}  of $f$ if it satisfies  $\overline\mu(x-a)=\epm(f)$ for some extension $\omu$ of  $\mu$ to $\kbx$.

\begin{Definition} \label{absKP}
	A monic $Q\in\kx$ is said to be an \textbf{abstract key polynomial} for $\mu$ if for every non-constant $f\in\kx$ we have
	\[
	\deg f<\deg Q\Longrightarrow \epm(f)<\epm(Q).
	\]
	We denote by $\Psi(\mu)$ the set of all abstract key polynomials for $\mu$.
\end{Definition}

For every polynomial $q\in\kx\setminus K$ consider the truncated function $\mu_q$ on $\kx$, defined as follows  on $q$-expansions:
\[
f=\sum\nolimits_{i\ge0}f_i\, q^i,\ \deg f_i<\deg q\ \Longrightarrow\ \mu_q(f)=\min\{\mu(f_i\,  q^i)\mid i\ge0\}.
\]

For every abstract key polynomial $Q$, the truncated function $\mu_Q$ is a valuation such that $\mu_Q\le\mu$ [\ref{Nova Spiva key pol pseudo convergent}, Proposition 2.6].

The {\bf graded algebra} of $\mu$ is the integral domain $\ggm=\bigoplus_{\alpha \in \Gamma_\mu}\mathcal P_\alpha/\mathcal P_\alpha^+$, where
\[
\mathcal P_\alpha=\{f\in K[x]\mid \mu(f)\geq \a\}\supseteq\mathcal P_\alpha^+=\{f\in K[x]\mid \mu(f)> \a\}.
\]

Set ${\rm supp}(\mu)=\{f\in K[x]\mid v(f)=\infty\}$. Every $f\in K[x]\setminus {\rm supp}(\mu)$ has a homogeneous \textbf{initial coefficient}  $\inm f\in\ggm$, defined as the image of $f$ in $\mathcal P_{\mu(f)}/\mathcal P_{\mu(f)}^+\sub \ggm$.

\begin{Definition}\label{keypolyno}
	A monic $\phi\in K[x]$ is a \textbf{key polynomial}  for $\mu$ if  $(\inm\phi)\ggm$ is a homogeneous prime ideal containing no initial coefficient \ $\inm f$ with  $\deg f< \deg \phi$.
\end{Definition}

We denote by $\kpm$ the set of all key polynomials for $\mu$. These polynomials are  necessarily irreducible in $\kx$.

\subsection{Characterizations of the $j$-invariant in terms of key polynomials}
Let $\mu$ be any valuation on $K[x]$. If $Q$ is an abstract key polynomial for $\mu$, then we define the $j$-invariant of $Q$ as 
\[ j(Q):= j_{\mu_Q}(Q).  \]
The next easy result allows us to connect the $j$-invariant with the objects studied in \ref{Nova Spiva key pol pseudo convergent}. 
\begin{Proposition}\label{Mainporop}
	$j(Q)\in I_\mu(Q)$.
\end{Proposition}
\begin{proof}
	This follows from [\ref{Novacoski key poly and min pairs}, Proposition 3.1]. Indeed, in the proof of [\ref{Novacoski key poly and min pairs}, Proposition 3.1] it is shown that if $r$ is the number of optimizing roots of $Q$, then $r\in I_\mu(Q)$. On the other hand, it follows from the definition that a root $a'$ of $Q$ is an optimizing root if and only if $v(a-a')\geq \epsilon(Q)$. Hence $j(Q)\in I_\mu(Q)$.
\end{proof}

\begin{Remark}
	It follows from Theorem \ref{Thm j = [K(a):ICF] min pair of defn} and the remark following Theorem \ref{Thm bound of ICF} that 
	\[ j(Q) = p^s \text{ for some } s\in\NN.  \]
\end{Remark}

\begin{Corollary}
	Assume that $Q$ and $Q'$ are two key polynomials for $\mu$, of the same degree, for which $\epsilon(Q)<\epsilon(Q')$. Then
	\[
	j(Q)\geq j(Q').
	\]
\end{Corollary}
\begin{proof}
	It follows from [\ref{Nart Novacoski min limit key pols}, Lemma 3.3] and Proposition \ref{Mainporop}.
\end{proof}

\begin{Remark}
	Every valuation-transcendental valuation can be realized as $\mu_Q$ for some key polynomial $Q$ as above.
\end{Remark}

For each $n\in\NN$ set $\Psi_n(\mu)$ to be the set of abstract key polynomials for $\mu$ of degree $n$. Since for every $Q\in \Psi_n(\mu)$ we have $\mu_Q\leq \mu$, and $\{\nu\mid \nu\leq \mu\}$ is totally ordered, we can talk about final segments, cofinal subsets, etc..

\begin{Remark}
	The number $b(Q):= \min I_\mu(Q)$ plays a crucial role in Spivakovsky's approach to the problem of Resolution of Singularities. 
\end{Remark}

\begin{Proposition}\label{Propostablep}
	Assume that $\{\mu_Q\mid Q\in \Psi_n(\mu)\}$ does not have a maximum. Then there exist $Q_0\in\Psi_n$ and $s\in\NN$ such that
	\[
	b(Q)=j(Q)=p^{s}\mbox{ for every }Q\in \Psi_n\mbox{ with }\mu(Q)\geq \mu(Q_0).
	\]
\end{Proposition}
\begin{proof}
	It follows from [\ref{Nart Novacoski min limit key pols}, Corollary 3.4] that there exists $Q_0\in \Psi_n$ and $s\in \NN$ such that
	\[
	I(Q)=I(Q_0)=\{p^s\}\mbox{ for every }Q\in \Psi_n\mbox{ with }\mu(Q)\geq \mu(Q_0).
	\]
	Hence, our result follows from Proposition \ref{Mainporop}.
\end{proof}

Take an element $a\in \overline K\setminus K$ and consider the valuation (with non-trivial support) defined by
\[
v_a(f)=v(f(a))\mbox{ for every }f\in K[x].
\]
Fix an Okutsu sequence
\[  A_0 (= \{a\}), A_1, \dotsc , A_n,  \]
for $a$. It follows from [\ref{Nart Novacoski Depth of extns of valns}, Theorem 1.4] that for a given $n\in \NN$ the set $\Psi_n(v_a)$ is non-empty if and only if $n=\deg_K(A_i)$ for some $i$. Moreover, the set $A_i$ is a singleton if and only if $\Psi_n(v_a)$ has a maximum.

An interesting consequence of the above result is the following.
\begin{Proposition}
	Suppose that $a\in\overline K\setminus K$ is such that for every $n\in\NN$, $n\leq \deg_K(a)$, the set $\Psi_n(v_a)$ does not have maximum. Then there exists a \emph{complete sequence of abstract key polynomials} $\textbf{Q}$ for $\theta$ such that
	\[
	b(Q)=j(Q)\in \{p^s\mid s\in \NN\}\mbox{ for every }Q\in \textbf{Q}.
	\]  
\end{Proposition}

For a given $n\in\NN$, the condition that $\{\mu_Q\mid Q\in \Psi_n(\mu)\}$ has or not a last element is very well understood. For instance, it follows from [\ref{Vaquie key pols}] that if $(K,v)$ is henselian, then the element $\theta$ is defectless if and only if $\{\mu_Q\mid Q\in \Psi_n(\mu)\}$ ($\mu=v_a$) has a maximum for every $n\in\NN$, $n<\deg_K(a)$. On the other hand, if an element has the same degree as its defect, i.e., for $a\in \overline{K}$ and $L:=K(a)$ we have
\[
d(L|K,v)=[L:K],
\]
then every set $\{\mu_Q\mid Q\in \Psi_n(\mu)\}$ does not have a maximum.

\pars In particular we have the following corollary.
\begin{Corollary}
	For $a\in \overline{K}$ and $L:=K(a)$, assume that
	\[
	d(L|K,v)=[L:K].
	\]
	Then $a$ admits a \emph{complete sequence of key polynomials $\textbf{Q}$} such that for every $Q\in \textbf{Q}$ we have
	\[
	b(Q)=j(Q)\in \{p^s\mid s\in \NN\}.
	\]
\end{Corollary}

\subsection{Graded algebra interpretation}
A way of interpreting the $j$-invariant is the following. For a valuation $\mu$ on $K[x]$, fix an extension $\overline \mu$ of $\mu$ to $\overline K[x]$, an abstract key polynomial $Q$ for $\mu$ such that $\mu=\mu_Q$ and an optimizing root $a$ of $Q$. We consider the graded algebra ${\rm gr}_{\overline \mu}(\overline K[x])$ as it has been done in [\ref{Vaquie key pols}] and [\ref{Dutta invariant of valn tr extns and connection with key pols}], for instance. As shown in \ref{Novacoski MVKP}, and observed in \ref{Decaup Spiva Mahboub ABKP comparison MVKP}, we have
\[
{\rm gr}_{\overline \mu}(\overline K[x])=\mathcal G[Y]
\]
where $\mathcal G$ is the abelian group generated by $\{{\rm in}_{\overline \mu} c\mid c\in \overline K^*\}$ and $Y={\rm in}_{\overline \mu}(x-a)$. Hence, it makes sense to talk about $\deg_{Y}\left({\rm in}_{\overline{\mu}}(f)\right)$ for any $f\in K[x]\setminus {\rm supp}(\mu)$.

Let $a_1,\ldots, a_n\in\overline K$ be all the roots (not necessarily distinct) of $Q$. Then
\begin{equation}\label{decompositonoff}
	Q(x)=\prod_{i=1}^n (x-a_i)=\prod_{i=1}^r (x-a_i)\prod_{i=r+1}^n (x-a_i)
\end{equation}
with
\[
\overline\mu(x-a_1)=\overline\mu(x-a_r)=\epsilon(Q)>\overline\mu(x-a_\ell)\mbox{ for every }\ell>r.
\]
\begin{Proposition}
	$j(Q)=\deg_Y\left({\rm in}_{\overline \mu}(Q)\right)=r$.
\end{Proposition}
\begin{proof}
	It follows directly from \eqref{decompositonoff}. In that case
	\[
	{\rm in}_{\overline \mu}(x-a_i)\in \mathcal G\mbox{ if and only if }i> r.
	\]
	For $i\leq r$ we have
	\[
	\deg_Y\left({\rm in}_{\overline \mu}(x-a_i)\right)=1.
	\]
\end{proof}

\begin{Remark}
	The above proposition also follows from [\ref{Dutta invariant of valn tr extns and connection with key pols}, Proposition 3.5].
\end{Remark}

\subsection{Newton polygon interpretation}
If we consider the Newton polygon $\Delta_f$ of the following points
\[
\{(b,\partial_bf)\mid 0\leq b\leq \deg(f)\},
\]
then we clearly obtain that $-\epsilon(f)$ is the slope of the first side of this polygon. Also, one can see that the length of this side is $\max I(f)$. Hence, if $I(Q)$ is a singleton (this will be the case for key polynomials \enquote{far enough in a $\Psi_n(\mu)$}, if $\Psi_n(\mu)$ does not have maximum), then $j(Q)$ is the length of the first side of $\Delta_Q$.

	\end{document}